\documentclass[reqno]{amsart}
\usepackage{amscd}

\usepackage{graphicx}
\usepackage[usenames,dvipsnames,svgnames]{xcolor}

\usepackage{amsaddr}

\usepackage[shortlabels]{enumitem} 

\usepackage[all]{xy}
\usepackage{tikz}
\usepackage{tikz-cd}
\usetikzlibrary{arrows,shapes,chains,matrix,positioning,scopes,cd,patterns}

\usepackage{todonotes}

\usepackage[top=30truemm,bottom=30truemm,left=25truemm,right=25truemm]{geometry}

\theoremstyle{plain}
\newtheorem{thm}{Theorem}[section]
\newtheorem{lem}[thm]{Lemma}

\newtheorem{pro}[thm]{Proposition}

\newtheorem{thm*}{Theorem}
\theoremstyle{definition}

\numberwithin{equation}{section}

\DeclareMathOperator{\Hom}{Hom}
\DeclareMathOperator{\Aut}{Aut}
\DeclareMathOperator{\End}{End}
\DeclareMathOperator{\Ext}{Ext}

\DeclareMathOperator{\Tor}{Tor}

\DeclareMathOperator{\Mod}{\mathsf{Mod}}
\DeclareMathOperator{\add}{\mathsf{add}}

\DeclareMathOperator{\Fd}{\mathsf{fd}}

\newcommand{\xto}[1]{\xrightarrow{#1}}

\renewcommand{\bar}{\overline}

\DeclareMathOperator{\op}{op}

\DeclareMathOperator{\kD}{\mathit{D}}

\DeclareMathOperator{\Ker}{Ker} 
\DeclareMathOperator{\Coker}{Coker}
\renewcommand{\Im}{\mathop{\mathrm{Im}}}

\DeclareMathOperator{\Ann}{Ann}

\DeclareMathOperator{\Dim}{\underline{dim}}

\DeclareMathOperator{\RHom}{{\bf R}Hom}

\newcommand{\ideal}[1]{\left\langle #1 \right\rangle}

\DeclareMathOperator{\id}{id} 
\renewcommand{\epsilon}{\varepsilon}

\newcommand{\al}{\alpha}

\newcommand{\ep}{\epsilon}
\newcommand{\si}{\sigma}

\newcommand{\la}{\lambda}
\newcommand{\Ga}{\Gamma}

\newcommand{\mbL}{\mathbb{L}}

\newcommand{\mbN}{\mathbb{N}}

\newcommand{\mbZ}{\mathbb{Z}}

\newcommand{\mcA}{\mathcal{A}}
\newcommand{\mcB}{\mathcal{B}}
\newcommand{\mcC}{\mathcal{C}}

\newcommand{\mcE}{\mathcal{E}}
\newcommand{\mcF}{\mathcal{F}}

\newcommand{\mcI}{\mathcal{I}}

\newcommand{\mcM}{\mathcal{M}}

\newcommand{\mcR}{\mathcal{R}}
\newcommand{\mcS}{\mathcal{S}}

\newcommand{\mcZ}{\mathcal{Z}}

\newcommand{\sD}{\mathsf{D}}

\DeclareMathOperator{\gr}{gr}

\DeclareMathOperator{\Rep}{\mathsf{Rep}}
\DeclareMathOperator{\GL}{\mathsf{GL}}

\newcommand{\dbslash}{/\!\!/}

\begin{document}
\title{On deformed preprojective algebras}
\author{William Crawley-Boevey ${}^1$}

\address{
Fakult\"at f\"ur Mathematik, Universit\"at Bielefeld, 33501 Bielefeld, Germany}
\email{wcrawley@math.uni-bielefeld.de}

\author{Yuta Kimura}
\address{
Graduate School of Mathematical Sciences, The University of Tokyo, 3-8-1 Komaba Meguro-ku Tokyo 153-8914, Japan}
\email{ykimura@ms.u-tokyo.ac.jp}





\thanks{${}^1$ Corresponding author.}

\thanks{The authors have been supported by the Alexander von Humboldt Foundation in the framework of an Alexander von Humboldt Professorship endowed by the German Federal Ministry of Education and Research.}

\begin{abstract}
Deformed preprojective algebras are generalizations of the usual preprojective algebras introduced by Crawley-Boevey and Holland, which have applications to Kleinian singularities, the Deligne-Simpson problem, integrable systems and noncommutative geometry. In this paper we offer three contributions to the study of such algebras: (1) the 2-Calabi-Yau property; (2) the unification of the reflection functors of Crawley-Boevey and Holland with reflection functors for the usual preprojective algebras; and (3) the classification of tilting ideals in 2-Calabi-Yau algebras, and especially in deformed preprojective algebras for extended Dynkin quivers.
\end{abstract}
\maketitle
\section{Introduction}
Deformed preprojective algebras were introduced in \cite{CBH98} in order to study noncommutative deformations of Kleinian singularities; they have been studied further in \cite{CB-Geo}, and used in \cite{CB-OM} to solve an additive analogue of the Deligne-Simpson problem. They also have a role in integrable systems, see e.g.~\cite{BCE}, and noncommutative geometry, see e.g.~\cite{BGK}.
In this paper we further develop the theory of these algebras. 

The definition is as follows. Let $K$ be a field and let $Q$ be a quiver with vertex set $Q_0$ and arrow set $Q_1$. We assume always that $Q$ is finite, and write $h(a)$ and $t(a)$ for the head and tail vertices of an arrow $a$. 
The \emph{double} $\overline{Q}$ of $Q$ is obtained by adjoining a reverse arrow $a^*:j\to i$ for each arrow $a:i\to j$ in $Q$.
We use the notation $(a^\ast)^{\ast}=a$, and for an arrow $a\in \bar{Q}_1$, we set $\ep(a)=1$ if $a\in Q_1$ and $\ep(a)=-1$ otherwise.
Let $I = Q_0$ and fix a \emph{weight} $\lambda\in K^I$. The corresponding \emph{deformed preprojective algebra} is
\[
\Pi^\lambda(Q) := K\overline{Q}/ (\rho_\lambda)
\quad\text{where}\quad 
\rho = \sum_{a\in \bar Q_1} \ep(a) a a^*
\quad\text{and}\quad 
\rho_\lambda = \rho - \sum_{i\in I} \lambda_i e_i.
\]
Here $K\overline{Q}$ denotes the path algebra of $\overline{Q}$. We use the convention that the path $ba$ exists if $h(a)=t(b)$, and the trivial path at vertex $i$ is denoted $e_i$. The usual (undeformed) preprojective algebra is $\Pi^0(Q)$.

Given a ring $R$, we usually consider left $R$-modules, and write $\Mod R$ for the corresponding category. 
Given an algebra $A$, we consider $A\otimes_K A$ as an $A$-bimodule using the outer $A$-actions, so $b(a\otimes a')c = ba\otimes a'c$. If $M$ is an $A$-bimodule, the space $\Hom_{AA}(M,A\otimes_K A)$ of $A$-bimodule homomorphisms becomes an $A$-bimodule using the inner $A$-actions, that is, for $f\in \Hom(M,A\otimes_K A)$ we have $(bfc)(m) = \sum_\lambda a_\lambda c \otimes b a'_\lambda$ where $f(m) = \sum_\lambda a_\lambda \otimes a'_\lambda$.
The algebra $A$ is said to be \emph{homologically smooth} if $A$ has a finite projective resolution by finitely generated $A$-bimodules, so $A$ is isomorphic to a perfect complex in the derived category of $A$-bimodules, and it is
said to be $d$-Calabi-Yau, for a natural number $d$, if it is homologically smooth and $\RHom_{AA}(A,A\otimes_K A) \cong A[-d]$ in that same derived category \cite[Definition 3.2.3]{Gin06}, \cite[Definition 7.2]{vdB15}. Our first result may already be known to specialists, but we didn't find a reference. 
After releasing our work as a preprint, Travis Schedler explained to us a different proof, based on the methods of \cite{KS}.

\begin{thm}\label{thm-ass-gr}
If $Q$ is a connected non-Dynkin quiver, then $\Pi^\lambda(Q)$ is 2-Calabi-Yau.
\end{thm}

This is proved in section \ref{sec-projres-CY}. If $Q$ is a Dynkin quiver, then $\Pi^0(Q)$ is a finite-dimensional self-injective algebra, so not in general 2-Calabi-Yau. If $Q$ is connected and non-Dynkin, it is known that $\Pi^0(Q)$ is 2-Calabi-Yau, see \cite{Boc08,CB-except,BBK02,GLS07}.

The \emph{dimension vector} of a finite-dimensional $\Pi^\lambda(Q)$-module $M$ is $\Dim M\in\mbZ^I$. We write $\epsilon_i\in\mbZ^I$ for the coordinate vector of a vertex $i$. The \emph{Ringel form} of $Q$ is the bilinear form 
$\ideal{-,-} : \mbZ^I \times \mbZ^I \to \mbZ$
given by
\begin{equation}\label{def-ringel-form-a}
\ideal{\alpha,\beta} := \sum_{i\in I}\alpha_i\beta_i-\sum_{a\in Q_1} \alpha_{t(a)}\beta_{h(a)}.
\end{equation}
We call $q(\al):=\langle \al, \al \rangle$ the \emph{quadratic form of $Q$},
and have a symmetric bilinear form $(\alpha,\beta):=\ideal{\alpha,\beta}+\ideal{\beta,\alpha}$.

We say that a vertex $i$ in $Q$ is \emph{loop-free} if there are no arrows with both head and tail at $i$, or equivalently $q(\epsilon_i)=1$. For such a vertex, the corresponding reflections $s_i\in\Aut(\mbZ^I)$ and $r_i\in\Aut(K^I)$ are defined by
\[
s_i\alpha:=\alpha-(\alpha,\ep_i)\ep_i,
\qquad
(r_i \la)_j=\la_j-(\ep_i,\ep_j)\la_i.
\]
For $\lambda\in K^I$ and $\alpha\in\mbZ^I$, let $\lambda\cdot\alpha=\sum_i\lambda_i\alpha_i$.
Then it is easy to see that $r_i(\la)\cdot\alpha=\la\cdot s_i(\alpha)$ for $\la,\alpha\in K^I$.
The \emph{Weyl group} $W$ is the group of automorphisms of $\mbZ^I$ generated by the simple reflections for loop-free vertices.
Then $W$ acts on $\mbZ^I$ and $K^I$ with $(w\la)\cdot \alpha = \la \cdot (w^{-1}\alpha)$ for $w\in W$, $\alpha\in\mbZ^I$ and $\la\in K^I$.
According to \cite{CBH98}, deformed preprojective algebras come equipped with \emph{reflection functors}
\[
E_i : \Mod\Pi^\lambda(Q) \to \Mod\Pi^{r_i(\lambda)}(Q)
\]
which exist when $i$ is a loop-free vertex and $\lambda_i\neq 0$. Moreover, if $M$ is finite-dimensional, then
\[
\Dim E_i(M) = s_i(\Dim M).
\]

The condition $\lambda_i\neq 0$ means that these reflection functors do not exist for $\Pi^0(Q)$. Instead there is a theory of tilting modules and functors, see \cite{Iyama-Reiten, BIRS,BK,BKT}. Our next results unify these two situations. Using cokernels and kernels, for any loop-free vertex $i$ we construct functors 
\[
C_i^\lambda, K_i^\lambda : \Mod \Pi^\lambda(Q) \to \Mod \Pi^{r_i\la}(Q)
\]
in section \ref{sec-ref-cokker}. Since $r_i(r_i\lambda) = \lambda$, it follows that $C_i^{r_i\lambda}$ and $K_i^{r_i\lambda}$ are functors in the reverse direction. Note that if $\lambda_i\neq 0$, then also $(r_i\lambda)_i \neq 0$. If $\lambda_i=0$, then $r_i\lambda = \lambda$, so $C_i^\lambda$ and $K_i^\lambda$ are functors from $\Mod \Pi^\lambda(Q)$ to itself.

If $i$ is a loop-free vertex and $\lambda_i=0$, there is a unique trivial simple module $S_i$ for $\Pi^\lambda(Q)$ with dimension vector the coordinate vector for $i$. We denote its annihilator by $I_i$. This is a 2-sided ideal in $\Pi^\lambda(Q)$ with $\Pi^\lambda(Q)/I_i\cong K$. Also $I_i = \Pi^\lambda(Q) (1-e_i) \Pi^\lambda(Q)$, from which it follows that $I_iI_i = I_i$.

\begin{thm}\label{intro-thm-CK}
If $i$ is a loop-free vertex, then
\begin{enumerate}[{\rm (i)}]
\item $C_i^\lambda$ is left adjoint to $K_i^{r_i\lambda}$.
\item If $\la_i\neq 0$, then $C_i^\lambda\cong K_i^\lambda\cong E_i$, and $C_i^\lambda$ is an equivalence of categories which satisfies $C_i^{r_i\lambda}C_i^\lambda \cong \id_{\Mod\Pi^{\la}(Q)}$. 
\item If $\la_i=0$, then $C_i^\lambda \cong I_i\otimes_{\Pi^{\la}(Q)}(-)$ and $K_i^\lambda \cong \Hom_{\Pi^{\la}(Q)}(I_i, -)$.
\end{enumerate}
\end{thm}

This is proved in section \ref{sec-ref-cokker}. In future we drop the superscript $\lambda$ and just write $C_i$ and $K_i$, leaving it to the reader to interpret appropriately. In section \ref{sec-braid} we prove the following.

\begin{thm}\label{intro-thm-iji-jij}
If $i$ and $j$ are loop-free vertices, then
\begin{enumerate}[{\rm (i)}]
\item $C_iC_j \cong C_jC_i$ and $K_i K_j \cong K_j K_i$ if there is no arrow between $i$ and $j$ in $Q$
\item $C_iC_jC_i \cong C_jC_iC_j$ and $K_iK_jK_i \cong K_jK_iK_j$ if there is exactly one arrow between $i$ and $j$ in $Q$.
\end{enumerate}
\end{thm}

We remark that part (ii) of the theorem is known in the case where $\la=0$ by \cite{BK}. (It can also can be shown by using relations on tilting ideals as in \cite{BIRS}.) On the other hand, our statement (ii) does not assume that $\la=0$.

Let $A$ be a $K$-algebra. Given a left $A$-module $M$, one writes $\add M$ for the full subcategory of $\Mod A$ consisting of the modules isomorphic to a direct summand of a finite direct sum of copies of $M$, so $\add A$ is the category of finitely generated projective left $A$-modules.
Let $n\ge 1$. Recall that an $A$-module $T$ is an \emph{$n$-tilting $A$-module} if it satisfies the following three conditions.
\begin{enumerate}
    \item There is an exact sequence $0 \to P_n \to \dots \to P_1 \to P_0 \to T \to 0$ with $P_0, P_1,\dots,P_n\in\add A$.
    \item $\Ext^i_A(T,T)=0$ for all $i>0$.
    \item There is an exact sequence $0 \to A \to T_0 \to T_1 \to \dots \to T_n \to 0$ with $T_0, T_1,\dots,T_n\in\add T$.
\end{enumerate}
In this paper, by a \emph{tilting module} we always mean a 1-tilting module. 
Note that by (1), any tilting module is finitely generated. One says that $T$ is \emph{partial tilting} if $T$ satisfies (1) (with $n=1$) and (2). One says that an ideal $I$ of $A$ is a \emph{tilting ideal} if $I$ is a tilting module as both left and right $A$-modules. 

Tilting ideals for $2$-Calabi-Yau algebras were studied in \cite{BIRS, Iyama-Reiten} under the assumption that the algebras are complete. In this paper we study such ideals for arbitrary $2$-Calabi-Yau algebras.

For an $A$-module $M$, let $\Ann_A(M)=\{a\in A \mid aM=0\}$ be the \emph{annihilator ideal} of $M$ in $A$. We say that a simple $A$-module $S$ is \emph{rigid} if $\Ext^1_A(S,S)=0$. 
Given a finite-dimensional rigid simple $A$-module $S$, we write $I_S$ for its annihilator ideal in $A$. In section~\ref{sec-tilt-prelim} we prove the following.

\begin{pro}
\label{pro-ial-props}
If $A$ is 2-Calabi-Yau and $S$ is a finite-dimensional rigid simple $A$-module, then $I_S$ is a tilting ideal in $A$, it has finite codimension in $A$, and there is an isomorphism $\End_A(I_S)\cong A^{op}$, under which $a\in A$ corresponds to the homothety of right multiplication by $a$.
\end{pro}

Let $\mcS$ be a set of finite-dimensional rigid simple $A$-modules. In the category of finite-dimensional $A$-modules, we write $\mcE(\mcS)$ for the Serre subcategory generated by $\mcS$, so $\mcE(\mcS)$ consists of the finite-dimensional modules whose composition factors belong to $\mcS$. For a finite sequence $S_1,S_2,\dots,S_r$ of modules in $\mcS$, we consider the ideal $I_{S_1 S_2\dots S_r}=I_{S_1}I_{S_2}\cdots I_{S_r}$ in $A$. For the empty sequence we define $I_{\emptyset}=A$. We denote by $\mcI(\mcS)$ the set of all ideals of this form. The following result, proved in section~\ref{sec-tilt-ideals}, is an analogue of \cite[Theorem III.1.6]{BIRS}.

\begin{thm}
\label{thm-set-i-equiv}
Suppose that $A$ is 2-Calabi-Yau. Any element $I \in \mcI(\mcS)$ is a tilting ideal with $A/I\in\mcE(\mcS)$ and $\End_A(I)=A$. Conversely any partial tilting left ideal $I$ in $A$ with $A/I\in \mcE(\mcS)$ is in $\mcI(\mcS)$. If $I,I'\in \mcI(\mcS)$ are isomorphic as left modules, they are equal.
\end{thm}

\begin{pro}\label{pro-ideal-relation}
Suppose that $A$ is 2-Calabi-Yau and $S,T\in\mcS$.
\begin{enumerate}
    \item $I_S I_S=I_S$.
    \item $I_S I_T = I_T I_S$ if $\Ext^1_A(S,T)=0$.
    \item $I_S I_T I_S = I_T I_S I_T$ if $\Ext^1_A(S,T)$ is 1-dimensional as a right $\End_A(S)$-module and as a left $\End_A(T)$-module.
\end{enumerate}
\end{pro}

This and the next theorem are proved in section~\ref{sec-relns}. For simplicity (to avoid valued quivers, and because it is sufficient for $\Pi^\la(Q)$) we consider the case that $\mcS$ is \emph{split}, by which we mean that $\End_A(S)=K$ for all $S\in \mcS$. In this case, the Ext-quiver $Q(\mcS)$ has as vertices the isomorphism classes of elements $S\in\mcS$, and with $\dim_K\Ext^1_A(S,T)$ arrows from $S$ to $T$. For $A$ a 2-Calabi-Yau algebra this is the double of an acyclic quiver. The associated Coxeter group $W(\mcS)$ is generated by elements $\si_S$, one for each $S\in \mcS$ up to isomorphism, subject to the relations that $\si_S^2=1$ for any $S$, $\si_S \si_T=\si_T \si_S$ if there are no arrows from $S$ to $T$ and $\si_S \si_T  \si_S=\si_T \si_S \si_T$ if there is exactly one arrow from $S$ to $T$.

\begin{thm}
\label{thm-cox-bij}
Suppose that $A$ is 2-Calabi-Yau. If $\mcS$ is a set of finite-dimensional rigid simple $A$-modules which is split, then there is a bijection $W(\mcS) \to \mcI(\mcS)$ given by $w \mapsto I_{S_1 S_2\cdots S_r}$ for a reduced expression $\si_{S_1}\si_{S_2}\cdots \si_{S_r}$ for $w$.
\end{thm}

The previous results apply in particular to the algebra $A=\Pi^\la(Q)$ where $Q$ is a connected non-Dynkin quiver. Below we give the classification of finite-dimensional rigid simple modules in this case.

Recall that the \emph{simple roots} for $Q$ are the coordinate vectors $\ep_i$ of loop-free vertices $i$. An element of $\mbZ^I$ is a \emph{real root} if it is the image of a simple root under the action of the Weyl group.
The \emph{fundamental region} $F$ is the set of vectors $\alpha$ in $\mbN^I$ such that $\alpha\neq 0$, the support of $\alpha$ is connected and $(\alpha, \ep_i)\leq 0$ for any $i\in I$. An \emph{imaginary root} is an element of $\mbZ^I$ of the form $w\beta$ or $-w\beta$ for some $w\in W$ and $\beta \in F$.
A \emph{root} is a real or imaginary root.
It is standard that any root $\alpha$ is either \emph{positive}, meaning that it belongs to  $\mbN^I$, or \emph{negative}, meaning that it belongs to $(-\mbN)^I$.
It is easy to see that $q(s_i\al)=q(\al)$ holds.
Therefore $q(\al)=1$ if $\al$ is a real root, and $q(\al)\leq 0$ if $\al$ is an imaginary root.

For $\lambda\in K^I$, we define $\Sigma_\lambda^{\rm re}$ to be the set of positive real roots $\alpha$ with $\lambda\cdot\alpha=0$ and such that there is no decomposition $\alpha=\beta+\gamma+\dots$ as a sum of two or more positive roots with $\lambda\cdot\beta=\lambda\cdot\gamma=\dots$. 
It is the intersection of the set of real roots and the set $\Sigma_\lambda$ of \cite{CB-Geo}.
The following result is clear from \cite[Theorem 1.2]{CB-Geo} in case the base field $K$ is algebraically closed (see also \cite[Theorem 2]{CB-OM}), or for general $K$ by the argument of \cite[Theorems 1.8,  1.9]{CBS-Multi}---we omit the details.

\begin{pro}
\label{pro-rig-sim}
The map sending a module to its dimension vector gives a 1:1 correspondence between the isomorphism classes of finite-dimensional rigid simple $\Pi^\lambda(Q)$-modules and the elements of $\Sigma_\lambda^{\rm re}$.
The endomorphism algebra of any finite-dimensional rigid simple module is isomorphic to $K$.
The dimension vector of any finite-dimensional non-rigid simple module is a positive imaginary root.
\end{pro}

We can say more in case $Q$ is an extended Dynkin quiver, see \cite[Theorem III.1.6]{BIRS}. Let $\mcR$ be the set of all finite-dimensional rigid simple $\Pi^\la(Q)$-modules. The last two theorems are proved in section~\ref{sec-ext-dynkin}.

\begin{thm}
\label{thm-all-cof}
If $Q$ is an extended Dynkin quiver and $A = \Pi^\la(Q)$, then there are only finitely many isomorphism classes of finite-dimensional rigid simple $A$-modules, and $\mcI(\mcR)$ is the set of all tilting ideals of finite codimension in $\Pi^\la(Q)$.
\end{thm}

In order to understand the structure of $Q(\mcR)$ we have the following result. Here $\delta$ is the minimal positive imaginary root for an extended Dynkin quiver $Q$. The notation is as in \cite{CBH98}.

\begin{thm}\label{thm-sing-comp-a}
If $Q$ is extended Dynkin quiver and $\la\cdot\delta=0$, then $Q(\mcR)$ is the double of a disjoint union of extended Dynkin quivers. If in addition $K$ is an algebraically closed field of characteristic zero, then there is a bijection between the connected components of $Q(\mcR)$ and the singular points of the affine quotient variety $\Rep(\Pi^\la(Q), \delta) \dbslash \GL(\delta)$.
\end{thm}

\section{The PBW and 2-Calabi-Yau properties}\label{section-calabi-yau}
\subsection{Filtrations and the PBW property}
Let $A$ be a $K$-algebra.
A \emph{filtration} of $A$ is a family $\mcF=\{A_{\le i}\}_{i \geq 0}$ ($i\in\mbZ$) of $K$-subspaces $A_{\le i}$ of $A$ satisfying $A=\bigcup_{i\geq 0}A_{\le i}$, $A_{\le i}\subseteq A_{\le i+1}$ and $A_{\le i} A_{\le j} \subseteq A_{\le i+j}$ for any $i,j\geq 0$.
If $A$ admits a filtration, then we say that $A=\bigcup_{i\geq 0}A_{\leq i}$ is a \emph{filtered $K$-algebra}.
In this case, the associated graded $K$-algebra $\mathsf{gr}A$ is defined as follows:
\[
\gr_{\mcF}A = \gr A:=\bigoplus_{i\geq 0}A_{\le i}/A_{\le i-1},
\]
where $A_{\le -1}=0$.
The algebra $\gr A$ is a ($\mbN$-) graded $K$-algebra such that the $i$-th component is $A_{\le i}/A_{\le i-1}$.

If $A=\bigoplus_{i\geq 0}A_i$ is a graded $K$-algebra then $A_{\le i} = \bigoplus_{j\le i} A_j$ defines a filtration of $A$, and the associated graded algebra is isomorphic to $A$ as graded $K$-algebras.
More precisely, if $x=\sum_{i=0}^lx_i \in A$ with $x_i\in A_i$, then the isomorphism is described as follows:
\begin{align}\label{isom-graded-gr}
A \to \gr A, \qquad x \mapsto (x_i + A_{\leq i-1})_i. 
\end{align}

Let $A$ be a filtered $K$-algebra and let $M$ be an $A$-module.
A filtration of $M$ is a family $\mcM=\{M_{\le i}\}_{i\geq 0}$ of $K$-subspaces $M_{\le i}$ of $M$ satisfying $M=\bigcup_{i\geq 0}M_{\le i}$, $M_{\le i}\subset M_{\le i+1}$ and $A_{\le i} M_{\le j}\subset M_{\le i+j}$.
For filtered $A$-modules $M=\bigcup_{i\geq 0}M_{\le i}$ and $N=\bigcup_{i\geq 0}N_{\le i}$, a morphism of filtered modules of degree $j\in\mbZ$ is a morphism of $A$-modules $f : M \to N$ such that $f(M_{\le i})\subset N_{\le i+j}$ holds for any $i\geq 0$.
In this case, the morphism $f$ induces a morphism
\[
\gr f : \gr M \to \gr N
\]
of graded $(\gr A)$-modules of degree $j$.
We have two well-known and straightforward lemmas.

\begin{lem}\label{lem-grf-inj}
Let $M,N$ be filtered $A$-modules and $f : M \to N$ be a morphism of filtered modules of degree $j\geq 0$. If $\gr f$ is injective, then so is $f$.
\end{lem}

\begin{proof}
Let $x\in M$ such that $f(x)=0$.
There exists $i\geq 0$ such that $x\in M_{\le i}$.
Because $f(x)=0$, we have $(\gr f)(\bar{x})=\bar{f(x)}=0$ in $N_{\le i+j}/N_{\le i+j-1}$, where $\bar{x}\in M_{\le i}/M_{\le i-1}$.
Since $\gr f$ is injective, we have $\bar{x}=0$, that is, $x\in M_{\le i-1}$.
By using this argument inductively, we have $x \in M_{\le -1}=0$.
\end{proof}

Let $A$ and $B$ be two filtered $K$-algebras.
For a filtered left $A$-module  $M$ and a filtered right $B$-module $N$, we have the following filtration of $M\otimes_K N$:
\begin{align}\label{tensor-filt}
(M\otimes_K N)_{\le i}:=\sum_{i=j+k} M_{\le j}\otimes_K N_{\le k}.
\end{align}
In particular this induces a filtration on the $K$-algebra $A\otimes_K B^{\op}$, and then $M\otimes_K N$ is a filtered $A$-$B$-bimodule, by which we mean that it is a filtered $A\otimes_K B^{\op}$-module.

\begin{lem}\label{lem-grf-bi-inj}
Let $A$ and $B$ be two filtered $K$-algebras.
Let $M$ be a filtered left $A$-module and $N$ be a filtered right $B$-module.
Then $\gr (M\otimes_K N)$ is isomorphic to $\gr M \otimes_K \gr N$ as a $\gr A$-$\gr B$-bimodule.
If $M=A$ and $N=B$, then this is an isomorphism of graded algebras.
\end{lem}

\begin{proof}
A morphism $\phi$ from $\gr(M\otimes_K N)$ to $\gr M \otimes_K \gr N$ is given as follows: for $z=x\otimes y + (M\otimes_K N)_{\le i-1}\in \gr(M\otimes_K N)_i$ with $x\in M_{\le j}, y\in N_{\le k}$ for $j+k=i$, let $\phi(z)=(x+M_{\le j-1})\otimes(y+N_{\le k-1})\in (\gr M \otimes_K \gr N)_i$.
Then this assignment induces a morphism of $\gr A$-$\gr B$-bimodules.
By comparing $K$-bases of $M_{\le i}/M_{\le i-1}$, $N_{\le i}/N_{\le i-1}$ and their images under $\phi$, one can show that $\phi$ is an isomorphism.
\end{proof}

The deformed preprojective algebra $\Pi^\lambda(Q)$ has a filtration $0 = \Pi_{\le -1} \subseteq \Pi_{\le 0} \subseteq \Pi_{\le 1} \subseteq \dots \subseteq \Pi^\lambda(Q)$, where $\Pi_{\le n}$ is spanned by the paths in $\overline{Q}$ of length at most $n$. 
There is a natural map $\Pi^0(Q)\to \gr \Pi^\lambda(Q)$ and it is surjective \cite[Lemma 2.3]{CBH98}.

\begin{lem}\label{lem-pbw}
If $Q$ is connected and non-Dynkin, then the natural map $\Pi^0(Q)\to \gr \Pi^\lambda(Q)$ is an isomorphism.
\end{lem}

In other terminology, this says that $\Pi^\lambda(Q)$ is a \emph{PBW deformation} of $\Pi^0(Q)$.

\begin{proof}
If $Q$ is connected and non-Dynkin, it is known that $\Pi^0(Q)$ is Koszul, see \cite{EE07} (and \cite{MV96,BBK02,MOV06} for special cases). The lemma thus follows from \cite[Theorem A]{HVOZ}.
\end{proof}

\subsection{Projective resolution and Calabi-Yau property}
\label{sec-projres-CY}
We begin with results about a partial projective resolution of $\Pi^\lambda(Q)$. The results are already known for the undeformed preprojective algebra \cite[Lemma 1]{CB-except}, and although not written anywhere, it has long been known to the first author that they generalize to $\Pi^\lambda(Q)$, since they were further adapted to multiplicative preprojective algebras in \cite[section 3]{CBS-Multi}.

Let $\Pi=\Pi^\lambda(Q)$.
Let $P_0$ and $P_1$ be the following projective $\Pi$-bimodules,
\[
P_0 = \bigoplus_{i\in I}\Pi e_i \otimes_K e_i\Pi 
\quad\text{and}\quad
P_1 =  \bigoplus_{a\in \bar{Q}_1}\Pi e_{h(a)}\otimes_K e_{t(a)}\Pi.
\]
For any $i\in I$, we write $\eta_i$ for the element $e_i\otimes e_i$ in the $i$th summand of $P_0$, and for any arrow $a\in \overline{Q}_1$, we write $\eta_a$ for the element $e_{h(a)}\otimes e_{t(a)}$ in the $a$th summand of $P_1$.

\begin{pro}\label{pro-bi-resol-a}
There is an exact sequence
$P_0 \xto{f} P_1 \xto{g} P_0 \xto{h} \Pi\to 0$
of $\Pi$-bimodules, where the maps are given by
\[
f(\eta_i) = \sum_{\substack{a\in \bar{Q}_1 \\ h(a)=i}}\ep(a)(\eta_a a^* + a \eta_{a^*}), \quad
g(\eta_a) = a\eta_{t(a)} - \eta_{h(a)} a, \quad
h(\eta_i) = e_i.
\]
\end{pro}

\begin{proof}
We have $\Pi = K\overline{Q}/I$, where $I$ is the ideal in $K\overline{Q}$ generated by the elements $\rho_i - \lambda_i e_i$. Let $S = K\overline Q_0 = \bigoplus_{i\in I} K e_i$. Combining the exact sequence \cite[Theorem 10.1]{Sch85} for $\Omega_S(\Pi)$ with \cite[Theorem 10.3]{Sch85}, gives an exact sequence
\[
I/I^2 \xto{\alpha} \Pi\otimes_{K \overline{Q}} \Omega_S(K\overline{Q}) \otimes_{K \overline{Q}} \Pi \xto{\beta} \Pi \otimes_S \Pi \xto{\gamma} \Pi\to 0.
\]
Here $\alpha(I^2+x) = 1\otimes (x\otimes 1-1\otimes x)\otimes 1$ for $x\in I$,  $\beta(1\otimes \omega\otimes 1)$ is the image of $\omega\in \Omega_S(K\overline{Q}) \subseteq K\overline{Q}\otimes_S K\overline{Q}$ in $\Pi\otimes \Pi$, and $\gamma$ is the multiplication map.

We identify $\Pi \otimes_S \Pi$ with $P_0$, and then $\gamma$ is identified with the map $h$.
Let $V = K\overline{Q}_1$, the vector space spanned by the arrows in $\overline{Q}$, which is naturally an $S$-bimodule. By \cite[Theorem 10.5]{Sch85}, there is an isomorphism 
\[
K \overline{Q} \otimes_S V \otimes_S K\overline{Q}\to \Omega_S(K\overline{Q})
\]
sending $1\otimes a\otimes 1$ to $a\otimes 1 - 1 \otimes a$. Thus we can identify
\[
P_1 \cong \Pi \otimes_S V\otimes_S \Pi \cong \Pi\otimes_{K \overline{Q}} \Omega_S(K\overline{Q}) \otimes_{K \overline{Q}} \Pi 
\]
with $\eta_a$ corresponding to $1\otimes (a\otimes 1 - 1\otimes a)\otimes 1$. Now under the identification of $\Pi \otimes_S \Pi$ with $P_0$, the element $a\otimes 1 - 1\otimes a$ corresponds to 
\[
a\otimes e_{t(a)} - e_{h(a)}\otimes a = a\eta_{t(a)} - \eta_{h(a)} a.
\]
It follows that $\beta$ corresponds to the map $g$. Finally $I$ is generated as an ideal by the elements $\rho_i - \lambda_i e_i$, so there is a surjective map $\phi:P_0 \to I/I^2$ sending $\eta_i$ to $I^2+\rho_i - \lambda_i e_i$. Its composition with $\alpha$
sends $\eta_i$ to $1\otimes(\rho_i\otimes 1-1\otimes \rho_i)\otimes 1$ since $\lambda_i e_i\otimes 1 = \lambda_i e_i\otimes e_i = 1\otimes \lambda_ie_i$.
Now
\begin{align*}
\rho_i\otimes 1-1\otimes \rho_i
&= \sum_{\substack{a\in \bar{Q}_1 \\ h(a)=i}}\ep(a)(a a^* \otimes 1 - 1\otimes a a^*)
\\
&= \sum_{\substack{a\in \bar{Q}_1 \\ h(a)=i}}\ep(a)\left(a (a^* \otimes 1 - 1\otimes a^*) + (a\otimes 1 - 1\otimes a)a^*\right)
\end{align*}
which corresponds in $P_1$ to the same element as $f(\eta_i)$. Thus $\alpha\phi$ corresponds to $f$, giving the result.
\end{proof}

\begin{pro}\label{pro-cx-dual}
The complexes $P_0\xto{f} P_1\xto{g} P_0$ and
\[
\Hom_{\Pi \Pi}(P_0,\Pi\otimes_K \Pi) \xto{-g^*} \Hom_{\Pi \Pi}(P_1,\Pi\otimes_K \Pi) \xto{f^*} \Hom_{\Pi \Pi}(P_0,\Pi\otimes_K \Pi)
\]
are isomorphic as complexes of $\Pi$-bimodules.
\end{pro}

Note that the change of sign for $g^*$ is irrelevant for the truth of the proposition; it is introduced because it arises in the definition of 
$\RHom_{\Pi}(\Pi,\Pi\otimes_K\Pi) \cong \Hom^{\bullet}_{\Pi \Pi}(P^{\bullet},\Pi\otimes_K \Pi)$.

\begin{proof}
Clearly we have $\Hom_{\Pi \Pi} (\Pi e_i \otimes_K e_j \Pi, \Pi\otimes_K \Pi) \cong e_i \Pi\otimes_K \Pi e_j$. Thus there are isomorphisms
\[
\alpha:\Hom_{\Pi \Pi}(P_0,\Pi\otimes_K \Pi)\to P_0,
\quad
\alpha(\phi) = \sum_{i,\lambda} p'_{i\lambda} \eta_i p_{i\lambda},
\]
where $\phi(\eta_i) = \sum_\lambda p_{i\lambda}\otimes p'_{i\lambda}\in e_i\Pi\otimes_K \Pi e_i$, and
\[
\beta: \Hom_{\Pi \Pi}(P_1,\Pi\otimes_K \Pi)\to P_1,
\quad
\beta(\psi) = \sum_{a\in\overline{Q}_1} \epsilon(a) p'_{a\lambda} \eta_{a^*} p_{a\lambda},
\]
where $\psi(\eta_a) = \sum_\lambda p_{a\lambda}\otimes p'_{a\lambda}\in e_{h(a)} \Pi\otimes_K \Pi e_{t(a)}$.
Now we have a commutative diagram
\[
\begin{CD}
\Hom_{\Pi \Pi}(P_0,\Pi\otimes_K \Pi) @>-g^*>> \Hom_{\Pi \Pi}(P_1,\Pi\otimes_K \Pi) @>f^*>> \Hom_{\Pi \Pi}(P_0,\Pi\otimes_K \Pi)
\\
@V\alpha VV @V\beta VV @V \alpha VV
\\
P_0 @>f>> P_1 @>g>> P_0.
\end{CD}
\]
\end{proof}

Let $\kD$ be the duality $\Hom_K(-,K)$.

\begin{pro}\label{pro-dim-Ext}
For finite dimensional $\Pi$-modules $M$, $N$, we have $\Ext^1_\Pi(N,M) \cong \kD\Ext^1_\Pi(M,N)$ and
\begin{align}\label{eq-dim-Ext}
    \dim_K \Ext_{\Pi}^1(M,N) = \dim_K \Hom_{\Pi}(M, N) + \dim_K \Hom_{\Pi}(N, M) - (\Dim M, \Dim N).
\end{align}
\end{pro}

\begin{proof}
Since the sequence in Proposition \ref{pro-bi-resol-a} is split as a sequence of right $\Pi$-modules, it induces an exact sequence
\[
P_0 \otimes_\Pi M \to P_1 \otimes_\Pi M \to P_0 \otimes_\Pi M \to M\to 0
\]
which is the start of a projective resolution of $M$. Applying $\Hom_\Pi(-,N)$ gives a complex
\begin{equation}\label{eq-hom-cx}
0 \to \Hom_\Pi(P_0 \otimes_\Pi M,N) \to \Hom_\Pi(P_1 \otimes_\Pi M,N) \to \Hom_\Pi(P_0 \otimes_\Pi M,N)\to 0.
\end{equation}
Since $\Hom_\Pi (\Pi e_i \otimes_K e_j\Pi \otimes_\Pi M,N) \cong \Hom_K (e_j M, e_i N)$, we see that the spaces in the complex are finite dimensional, and the alternating sum of their dimensions is $(\Dim M,\Dim N)$. Now the cohomology in the first two places is $\Hom_\Pi(M,N)$ and $\Ext^1_\Pi(M,N)$. For $P$ a f.g.\ projective $\Pi$-bimodule, we have natural isomorphisms
\[
\kD\Hom_\Pi(P\otimes_\Pi M,N) 
\cong
\kD \Hom_{\Pi \Pi}(P,\Hom_K(M,N))
\cong
\kD \left( \kD M \otimes_\Pi \Hom_{\Pi \Pi}(P,\Pi\otimes_K \Pi)\otimes_\Pi N)\right)
\]
\[
\cong \Hom_\Pi(\Hom_{\Pi \Pi}(P,\Pi\otimes_K \Pi)\otimes_\Pi N,M).
\]
Using Proposition~\ref{pro-cx-dual}, we see that the dual of the complex (\ref{eq-hom-cx}) is isomorphic to the same complex, but with $M$ and $N$ exchanged. Thus the cohomology in the second two places of the complex (\ref{eq-hom-cx}) is isomorphic to $\kD\Ext^1_\Pi(N,M)$ and $\kD\Hom_\Pi(N,M)$. The result follows.
\end{proof}


\begin{thm}
\label{thm-inj-cy}
If $Q$ is a connected non-Dynkin quiver, then the map 
$f$ in Proposition \ref{pro-bi-resol-a} is injective, so the exact sequence given there is a projective resolution of $\Pi$ as a $\Pi$-bimodule, and the global dimension of $\Pi$ is at most 2.
\end{thm}

\begin{proof}
For $\lambda=0$ this follows from the fact, mentioned in the proof of Lemma~\ref{lem-pbw}, that $\Pi^0(Q)$ is Koszul with known Hilbert series. In general, to show the dependence on $\lambda$, we write the map $f$ in Proposition~\ref{pro-bi-resol-a} as $f^\lambda : P_0^\lambda \to P_1^\lambda$. Now the filtration on $\Pi$ induces filtrations on the bimodules $P_0^\lambda$ and $P_1^\lambda$, and there are natural maps $P_i^0 \to \gr P_i^\lambda$, which are isomorphisms thanks to Lemmas~\ref{lem-grf-bi-inj} and~\ref{lem-pbw}. We have a commutative square
\[
\begin{CD}
P^0_0 @>f^0>> P_1^0 \\
@VVV @VVV \\
\gr P^\lambda_0 @>\gr f^\lambda>> \gr P^\lambda_1
\end{CD}
\]
and since $f^0$ is injective, Lemma~\ref{lem-grf-inj} implies that $f^\lambda$ is injective. Tensoring the bimodule projective resolution of $\Pi$ with a $\Pi$-module $M$, the resulting sequence remains exact (since the projective resolution of $\Pi$ is split as a sequence of one-sided $\Pi$-modules), so shows that $M$ has a projective resolution of length at most 2.
\end{proof}

\begin{proof}[Proof of Theorem \ref{thm-ass-gr}]
By Theorem~\ref{thm-inj-cy}, in the derived category of $\Pi$-bimodules, $\Pi$ is isomorphic to the complex $P_0\to P_1\to P_0$, so the Calabi-Yau property follows from Proposition~\ref{pro-cx-dual}.
\end{proof}

\section{Reflection functors}\label{section-DPA}

In this section, we define reflection functors on the module categories of deformed preprojective algebras and observe relations on these functors. Since these functors changes weights, contrary to section~\ref{section-calabi-yau}, we denote by $\Pi^\la=\Pi^\la(Q)$ the deformed preprojective algebra of a quiver $Q$ with a weight $\la$.
\subsection{Reflection functors via cokernels and kernels}
\label{sec-ref-cokker}
In this subsection, we induce and study reflection functors via cokernels and kernels on modules over deformed preprojective algebras.

A representation $V$ of $\Pi^{\la}=\Pi^\la(Q)$ is a set $V=(V_i, V_a)$ of $K$-vector spaces $V_i$ ($i\in I$) and morphisms of $K$-vector spaces $V_a : V_{t(a)} \to V_{h(a)}$ ($a\in \bar{Q}_1$) satisfying the following relations for each $i\in I$:
\[
\sum_{ h(a)=i}\ep(a)V_aV_{a^{\ast}}=\la_i \id_{V_i}.
\]
Here the sum is over all arrows $a$ in $\overline{Q}$ with head at $i$.
A morphism between two representations $U, V$ is a set $f=(f_i)_{i\in I}$ of $K$-linear morphisms $f_i : U_i \to V_i$ satisfying $f_{h(a)}U_a = V_a f_{t(a)}$ for any $a\in\bar{Q}_1$.
We denote by $\Rep\Pi^\la$ the category of representations  of $\Pi^{\la}$.
Then it is well-known that there is an equivalence $F : \Mod\Pi^{\la} \to \Rep\Pi^\la$ of categories given by $F(M)_i=e_iM$ and $F(M)_a=(a\cdot)$.
So we identify representations of $\Pi^{\la}$ and $\Pi^{\la}$-modules by this equivalence.

Fix a loop-free vertex $i\in I$ and a representation $V=(V_j, V_a)$ of $\Pi^{\la}$, let \[V_{\oplus }=\bigoplus_{ h(a)=i}V_{t(a)},\]
and let $\mu_a^V$ and $\pi_a^V$ be a canonical injection and a canonical surjection between $V_{t(a)}$ and $V_{\oplus}$.
Let
\begin{align*}
\mu^V=\sum_{h(a)=i}\mu_a^V V_{a^{\ast}} : V_i \to V_{\oplus}\\
\pi^V=\sum_{h(a)=i}\ep(a)V_a\pi_a^V : V_{\oplus} \to V_i.
\end{align*}
So we have a sequence
\begin{align}\label{seq-theta-phi}
&V_i \xto{\mu^V} V_{\oplus} \xto{\pi^V} V_i.
\end{align}
It is easy to see that $\pi^V\mu^V=\la_i\id_{V_i}$ holds.
If $V$ is clear from the context, then we write $\pi^V=\pi$, $\mu^V=\mu$, etc.
For $b, c\in\bar{Q}_1$ with $h(b)=h(c)=i$, we have
\begin{equation}\label{eq-pinu}
\pi_c\mu\pi\mu_b=\ep(b)V_{c^{\ast}}V_b.
\end{equation}

We define a representation $C_i(V)=(C_i(V)_j, C_i(V)_a)$ of $\Pi^{r_i\la}$ as follows.
For $j\in I$, let
\[
  C_i(V)_j=\begin{cases}
    V_j & (j\neq i), \\
    \Coker(\mu) & (j=i).
  \end{cases}
\]
We have an exact sequence $V_i \xto{\mu} V_{\oplus} \xto{c} C_i(V)_i \to 0.$
Since $\pi\mu=\la_i\id_{V_i}$, there is a morphism $\gamma : C_i(V)_i \to V_{\oplus}$ which makes the following diagram commutative with an exact row:
\begin{equation}\label{dia-gamma}
\begin{tikzcd}
V_i \arrow[r, "\mu"] & V_{\oplus} \arrow[r, "c"] \arrow[d, "\mu\pi-\la_i\id_{V_{\oplus}}"'] & C_i(V)_i \arrow[ld, "\gamma"] \arrow[r] & 0 \\
& V_{\oplus}
\end{tikzcd}
\end{equation}
Then for an arrow $a\in \bar{Q}_1$, let
\begin{equation}\label{V'-hom}
  C_i(V)_a=\begin{cases}
    V_a & (h(a)\neq i \neq t(a)), \\
    \ep(a)c\mu_a & (h(a)=i), \\
    \pi_{a^{\ast}}\gamma & (t(a)=i).
  \end{cases}
\end{equation}
For a morphism $f: U \to V$ of representations of $\Pi^{\la}$,
let $C_i(f)_j=f_j$ if $j\neq i$ and  $C_i(f)_i$ be a map induced from the following commutative diagram:
\begin{equation}\label{diag-cok-f}
\begin{tikzcd}
U_i \arrow[r, "\mu^U"] \arrow[d, "f_i"] & U_{\oplus} \arrow[r, "c^U"] \arrow[d, "\oplus_{a}f_{t(a)}"] & C_i(U)_i \arrow[r] \arrow[d, "C_i(f)_i"] & 0 \\
V_i \arrow[r, "\mu^V"] & V_{\oplus} \arrow[r, "c^V"] & C_i(V)_i \arrow[r] & 0,
\end{tikzcd}
\end{equation}
where $\oplus_{a}f_{t(a)} := \oplus_{h(a)=i}f_{t(a)}$.

We see that $C_i$ is actually a functor.
It is known that $C_i$ is a functor in the case where $\la=0$ \cite{BK}.
If $\la\neq 0$, then this was stated in \cite{CBH98} without a detailed proof.
Therefore we write the proof here.
\begin{pro}\label{pro-C-functor}
For each loop-free vertex $i\in I$, the above $C_i$ gives a covariant functor
\[
C_i : \Mod\Pi^{\la} \to \Mod\Pi^{r_i\la}.
\]
\end{pro}
\begin{proof}
We first see that $V'=C_i(V)$ is a representation of $\Pi^{r_i\la}(Q)$.
We show that $\sum_{h(a)=j}\ep(a)V'_aV'_{a^{\ast}}=(r_i\la)_j\id_{V'_j}$ holds for each vertex $j$.
Assume that $j=i$, then we have
\begin{align*}
 \sum_{h(a)=i}\ep(a)V'_aV'_{a^{\ast}}c  \stackrel{(\ref{V'-hom})}{=} \sum_{h(a)=i} \ep(a)\ep(a)c\mu_a\pi_a\gamma c
  \stackrel{(\ref{dia-gamma})}{=} \sum_{h(a)=i} c\mu_a\pi_a(\mu\pi-\la_i\id_{V_{\oplus i}}) 
  = c(\mu\pi-\la_i\id_{V_{\oplus i}})
  =-\la_ic.
\end{align*}
Since $c$ is surjective, we have the desired equality.

Assume that $j\neq i$ and there are $m$ arrows between $i$ and $j$ in $Q$.
We have $(r_i\la)_j=\la_j+m\la_i$ and
\begin{align*}
 \sum_{h(a)=j,\,t(a)\neq i}\ep(a)V'_aV'_{a^{\ast}} + \sum_{\substack{h(a)=j \\ t(a)=i}}\ep(a)V'_aV'_{a^{\ast}} 
 & \stackrel{(\ref{V'-hom})}{=}\sum_{\substack{h(a)=j \\ t(a)\neq i}} \ep(a)V_aV_{a^{\ast}} + \sum_{\substack{h(a)=j \\  t(a)= i}}\ep(a)\pi_{a^{\ast}}\gamma\ep(a^{\ast}) c\mu_{a^{\ast}}\\
 & \stackrel{(\ref{dia-gamma})}{=} \sum_{\substack{h(a)=j \\ t(a)\neq i}} \ep(a)V_aV_{a^{\ast}} - \sum_{\substack{h(a)=j \\  t(a)= i}} \pi_{a^{\ast}}\mu\pi\mu_{a^{\ast}} + \la_i\sum_{\substack{h(a)=j \\  t(a)= i}} \pi_{a^{\ast}}\mu_{a^{\ast}} \\
 & \stackrel{(\ref{eq-pinu})}{=}(\la_j+m\la_i)\id_{V_j}.
\end{align*}
Therefore $V'=C_i(V)$ is a representation of $\Pi^{r_i\la}(Q)$.

Let $f :U \to V$ be a morphism of representations of $\Pi^{\la}$.
We check that $f'=C_i(f)$ is a morphism of representations of $\Pi^{r_i\la}$.
Let $b\in \bar{Q}_1$.
If $h(b)=i$, then
\[
    f'_i U'_b
    \stackrel{(\ref{V'-hom})}{=} f'_i\ep(b)c^U\mu_b^U  
    \stackrel{(\ref{diag-cok-f})}{=} \ep(b)c^V(\oplus_af_{t(a)})\mu_b^U
    =\ep(b)c^V\mu_b^V f_{t(b)}
    \stackrel{(\ref{V'-hom})}{=} V'_{b}f_{t(b)}
\]
holds, where $\oplus_{a}f_{t(a)} = \oplus_{h(a)=i}f_{t(a)}$.

Assume that $t(b)=i$.
For any arrow $c$ of $\bar{Q}$ with $h(c)=i$, we have
\begin{align*}
\pi_{b^{\ast}}^V\mu^V\pi^V(\oplus_a f_{t(a)})\mu_c^U
= \pi_{b^{\ast}}^V\mu^V\pi^V\mu_c^V f_{t(c)}
\stackrel{(\ref{eq-pinu})}{=} \ep(c)V_bV_c f_{t(a)}
= \ep(c)f_{h(b)}U_b U_c
\stackrel{(\ref{eq-pinu})}{=} f_{h(b)}\pi_{b^{\ast}}^U\mu^U\pi^U\mu_c^U.
\end{align*}
Since $\oplus_c \mu_c^U = \id_{U_{\oplus}}$ holds, we have $\pi_{b^{\ast}}^V\mu^V\pi^V(\oplus_a f_{t(a)}) = f_{h(b)}\pi_{b^{\ast}}^U\mu^U\pi^U$.
Therefore we have
\begin{align*}
    V'_b f'_i c^U & \stackrel{(\ref{V'-hom})}{=} \pi_{b^{\ast}}^V \gamma^V f'_i c^U \stackrel{(\ref{diag-cok-f})}{=} \pi_{b^{\ast}}^V \gamma^V c^V (\oplus_a f_{t(a)}) \stackrel{(\ref{dia-gamma})}{=} \pi_{b^{\ast}}^V (\mu^V\pi^V-\la_i\id_{V_{\oplus}})(\oplus_a f_{t(a)}) \\
    & = f_{h(b)}\pi_{b^{\ast}}^U(\mu^U\pi^U-\la_i\id_{V_{\oplus}}) \stackrel{(\ref{dia-gamma})}{=} f_{h(b)}\pi_{b^{\ast}}^U\gamma^U c^U \stackrel{(\ref{V'-hom})}{=} f_{h(b)} U'_{b} c^U.
\end{align*}
Since $c^U$ is surjective, we have the desired equality.
If $h(b)\neq i \neq t(b)$, then it is easy to see that $V'_{h(b)}f'_b=f'_bV'_{t(b)}$ holds.
Thus $f'$ is a morphism.

Finally, $C_i(gf)=C_i(g)C_i(f)$ holds by the diagram (\ref{diag-cok-f}).
This completes the proof.
\end{proof}
For a loop-free vertex $i\in I$, since $r_i r_i=\id$, by the same construction, we also have a covariant functor from $\Mod\Pi^{r_i\la}$ to $\Mod\Pi^{\la}$.
Therefore, to distinguish directions of functors, we use the superscript $\la$,
\[C_i^{\la} = C_i : \Mod\Pi^{\la} \to \Mod\Pi^{r_i\la}.\]

Next we define reflection functors via kernels.
For a representation $V=(V_i, V_a)$ of $\Pi^{\la}$ and a loop-free vertex $i\in I$, we define a representation $K_i^\la(V)=(K_i^\la(V)_j, K_i^\la(V)_a)$ as follows:
For $j\in I$, let
\[
  K_i(V)_j=\begin{cases}
    V_j & (j\neq i), \\
    \Ker(\pi) & (j=i).
  \end{cases}
\]
We have a short exact sequence $0 \to K_i^\la(V)_i \xto{\iota} V_{\oplus} \xto{\pi} V_i.$
Since $\pi\mu=\la_i\id_{V_i}$, there is a morphism $\gamma : V_{\oplus}\to K_i^\la(V)_i$ which makes the following diagram commutative:
\begin{equation}\label{dia-gamma-kernel}
\begin{tikzcd}
0 \arrow[r] & K_i^\la(V)_i \arrow[r, "\iota"] & V_{\oplus} \arrow[r, "\pi"] & V_i \\
& & V_{\oplus} \arrow[u, "\mu\pi-\la_i\id"'] \arrow[ul, "\gamma"]
\end{tikzcd}
\end{equation}
Recall that $V_{t(a)} \xto{\mu_a} V_{\oplus} \xto{\pi_a} V_{t(a)}$ are the canonical inclusion and  surjection.
Then for an arrow $a\in \bar{Q}_1$, let
\[
  K_i^\la(V)_a=\begin{cases}
    V_a & (h(a)\neq i \neq t(a)), \\
    \ep(a)\gamma\mu_a & (h(a)=i), \\
    \pi_{a^{\ast}}\iota & (t(a)=i).
  \end{cases}
\]

\begin{pro}\label{pro-K-functor}
For each loop-free vertex $i\in I$, we have a covariant functor $K_i^\la: \Mod\Pi^{\la} \to \Mod\Pi^{r_i\la}$.
\end{pro}
\begin{proof}
The proof is similar for Proposition \ref{pro-C-functor}.
\end{proof}

\begin{proof}[Proof of Theorem \ref{intro-thm-CK}]
(i)
Let $V=(V_i, V_a)$ be a representation of $\Pi^{\la}$ and $W=(W_i, W_a)$ be a representation of $\Pi^{r_i\la}$.
Let $f: C_i^{\la}(V) \to W$ be a morphism of $\Pi^{r_i\la}$-modules.
Then we have the following commutative diagram with exact rows and an induced morphism $g_i$:
\begin{equation*}
\begin{tikzcd}
& V_i \arrow[r, "\mu"] \arrow[d,"g_i", dashed] & V_{\oplus} \arrow[r,"c"] \arrow[d, "\oplus f_{t(a)}"] & C_i^{\la}(V)_i \arrow[r] \arrow[d,"f_i"] & 0\\
0 \arrow[r] & K_i^{r_i\la}(W)_i \arrow[r, "\iota"] & W_{\oplus} \arrow[r, "\pi"] & W_i
\end{tikzcd}
\end{equation*}
Let $g : V \to K_i^{r_i\la}(W)$ be a map given by $g_j=f_j$ if $j\neq i$ and $g_i$ as in the diagram.
Then this $g$ is a morphism of $\Pi^{\la}$-modules.
We have a map $\Hom_{\Pi^{r_i\la}}(C_i^{\la}(V), W) \to \Hom_{\Pi^{\la}}(V, K_i^{r_i\la}(W))$ by $f\mapsto g$.
The dual argument induces an inverse direction map and they are inverse each other.
Therefore $C_i^\la$ is left adjoint to $K_i^{r_i \la}$.

(ii)
Assume that $\la_i\neq 0$.
Then the sequence (\ref{seq-theta-phi}) splits, so the functors $C_i^\la$ and $K_i^\la$ are naturally isomorphic.
By the definition of $E_i$, $E_i$ and $K_i^\la$ are isomorphic as functors, see \cite{CBH98}. 
Therefore $C_i^\la \cong K_i^\la \cong E_i$ is an equivalence of categories.

(iii)
Use the same proof as in \cite[section 5]{BKT}; see also \cite[section 2]{BK}.
\end{proof}

\subsection{Braid relations on reflection functors}
\label{sec-braid}
In this subsection we prove Theorem \ref{intro-thm-iji-jij}.
We use the following technical lemma.
\begin{lem}\label{lem-ses-reduce}
Assume that there is the following exact sequence
\begin{align*}
V\oplus W \xto{\mathbf{A}}  V \oplus X \oplus Y \xto{(\rho\,\sigma\, \tau)} Z \to 0,
\qquad
\mathbf{A} = \scalebox{0.8}{$\begin{pmatrix}
\id & f \\
g & 0 \\
0 & h
\end{pmatrix}$}.
\end{align*}
Then it induces the following exact sequence
\begin{align*}
    W \xto{\mathbf{B}} X\oplus Y \xto{(\sigma\, \tau)} Z \to 0,
\qquad
\mathbf{B}=\scalebox{0.8}{$\begin{pmatrix}
-gf  \\
h  
\end{pmatrix}$}.
\end{align*}
\end{lem}
\begin{proof}
This follows from a direct calculation.
\end{proof}

The following theorem is a more precise statement of Theorem~\ref{intro-thm-iji-jij} for the functors $C_i$.

\begin{thm}\label{thm-iji-jij}
Let $i,j\in I$ be loop-free vertices.
We have the following isomorphisms of functors.
\begin{enumerate}[{\rm (a)}]
    \item $C_i^{r_j \la}C_j^{\la} \cong C_j^{r_i \la}C_i^{\la}$ if there is no arrow between $i$ and $j$ in $Q$.
    \item $C_i^{r_{j}r_{i} \la}C_j^{r_i \la}C_i^{\la} \cong C_j^{r_{i}r_{j} \la}C_i^{r_j \la}C_j^{\la}$ if there is exactly one arrow between $i$ and $j$ in $Q$.
\end{enumerate}
\end{thm}

Note that since $r_i$ is the dual of $s_i$, for loop-free vertices $i,j$ of $Q$, if there is no arrow between $i$ and $j$ in $Q$, then $r_ir_j = r_jr_j$, and if there is exactly one arrow between $i$ and $j$ in $Q$, then $r_i r_j r_i = r_j r_i r_j$.

\begin{proof}
For simplicity, we write
$C_{i_1i_2i_3}:=C_{i_1}^{r_{i_2}r_{i_3}\la}C_{i_2}^{r_{i_3}\la} C_{i_3}^{\la}$ for loop-free vertices $i_1,i_2,i_3$ of $Q$.
The isomorphism (a) is clear by the definition of $C_i$. 
We show the statement (b).
Assume that there is exactly one arrow $\alpha$ from $j$ to $i$ in $Q$.
Let $V=(V_i,V_a)$ be a representation of $\Pi^{\la}$.
We use the following notations of vector spaces:
\[
V^i:=\bigoplus_{\substack{h(a)=i \\ a\neq \alpha}}V_{t(a)},
\qquad
V^j:= \bigoplus_{\substack{h(a)=j \\ a\neq \alpha^{\ast}}}V_{t(a)}.
\]
Then the sequence (\ref{seq-theta-phi}) for $i$ and $j$ can be described as follows:
\[
V_i \xto{{}^t\scalebox{0.6}{$\begin{pmatrix}X^i & V_{\alpha^{\ast}}\end{pmatrix}$}} V^i\oplus V_j \xto{(Y^i\, V_{\alpha})} V_i,
\qquad
V_j \xto{{}^t(X^j V_{\alpha})} V^j\oplus V_i \xto{(Y^j\, -V_{\alpha^{\ast}})} V_j.
\]
By the definitions of $C_j(V)$ and $C_i(V)$, we have the following commutative diagrams with exact rows:
\begin{equation}\label{ex-j-i}
\begin{tikzcd}[baseline=0]
    V_j \arrow[rr, "{}^t(X^j\,V_{\alpha})"] & & V^j \oplus V_i \arrow[r, "(c_1\,c_1')"] \arrow[d, "\Phi"'] & C_j(V)_j \arrow[r] \arrow[dl, "{}^t(\gamma_1\,\gamma_1')"] & 0\\
&& V^j \oplus V_i
\end{tikzcd},
\qquad
   \begin{tikzcd}[baseline=0]
    V_i \arrow[rr, "{}^t(X^i\,V_{\alpha^{\ast}})"] & & V^i \oplus V_j \arrow[r, "(c_2\,c_2')"] \arrow[d, "\Psi"'] & C_i(V)_i \arrow[r] \arrow[dl, "{}^t(\gamma_2\,\gamma_2')"] & 0\\
&& V^i \oplus V_j
\end{tikzcd},
\end{equation}
where $\Phi={}^t(X^j\,\,V_{\alpha})(Y^j\,\,\,-V_{\alpha^{\ast}})-\la_j\id$ and
$\Psi={}^t(X^i\,\,V_{\alpha^{\ast}})(Y^i\,\,\,V_{\alpha})-\la_i\id$.

By the definition of $C_{ij}(V)$ and $C_{ji}(V)$, we have the following commutative diagrams with exact rows:
\begin{equation}\label{ex-ij-ji}
\begin{tikzcd}[baseline=0]
    V_i \arrow[r, "{}^t(X^i\,-c_1')"] & V^i \oplus C_j(V)_j \arrow[r, "(d_1\,d_1')"] \arrow[d, "\Phi'"'] & C_{ij}(V)_i \arrow[r] \arrow[dl, "{}^t(\delta_1\,\delta_1')"] & 0\\
    & V^i \oplus C_j(V)_j
\end{tikzcd},
\qquad
   \begin{tikzcd}[baseline=0]
    V_j \arrow[r, "{}^t(X^j\,c_2')"] & V^j \oplus C_i(V)_i \arrow[r, "(d_2\,d_2')"] \arrow[d, "\Psi'"'] & C_{ji}(V)_j \arrow[r] \arrow[dl, "{}^t(\delta_2\,\delta_2')"] & 0\\
    & V^j \oplus C_i(V)_i
\end{tikzcd},
\end{equation}
where $\Phi'={}^t(X^i\,\,-c_1')(Y^i\,\,\,\gamma_1')-(r_j\la)_i\id$ and
$\Psi'={}^t(X^j\,\,c_2')(Y^j\,\,-\gamma_2')-(r_i\la)_j\id$.

By the definitions of $C_{jij}(V)$ and $C_{iji}(V)$, we have the following commutative diagrams with exact rows:
\begin{equation}\label{ex-jij-iji}
\begin{tikzcd}[baseline=0]
    C_j(V)_j \arrow[r, "{}^t(\gamma_1\,d_1')"] & V^j \oplus C_{ij}(V)_i \arrow[r, "(e_1\,e_1')"] \arrow[d, "\Phi''"'] & C_{jij}(V)_j \arrow[r] \arrow[dl, "{}^t(\ep_1\,\ep_1')"] & 0\\
    & V^j \oplus C_{ij}(V)_i
\end{tikzcd},
\quad
   \begin{tikzcd}[baseline=0]
    C_i(V)_i \arrow[r, "{}^t(\gamma_2\,-d_2')"] & V^i \oplus C_{ji}(V)_j \arrow[r, "(e_2\,e_2')"] \arrow[d, "\Psi''"'] & C_{iji}(V)_i \arrow[r] \arrow[dl, "{}^t(\ep_2\,\ep_2')"] & 0\\
    & V^i \oplus C_{ji}(V)_j
\end{tikzcd},
\end{equation}
where $\Phi''={}^t(\gamma_1\,\,d_1')(c_1\,\,\,-\delta_1')-(r_{i}r_{j}\la)_j\id$ and
$\Psi''={}^t(\gamma_2\,\,-d_2')(c_2\,\,\,\delta_2')-(r_{j}r_{i}\la)_i\id$.

We show that $C_{iji}(V) \cong C_{jij}(V)$ as left $\Pi^{\la'}$-modules, where $\la'=r_{i}r_{j}r_{i}\la=r_{j}r_{i}r_{j}\la$.
By combining the exact rows of the left hand diagrams in (\ref{ex-j-i}) and (\ref{ex-ij-ji}), we have the following exact sequence:
\begin{align}\label{ex-V''i}
    V_i \oplus V_j \xto{\mathbf{A}} V_i \oplus V^i \oplus V^j \xto{(d_1'c_1'\,\,d_1\,\,d_1'c_1)} C_{ij}(V)_i \to 0,
    \quad
    \mathbf{A}=\scalebox{0.8}{$\begin{pmatrix}
    -\id & V_{\alpha} \\
    \mathrm{X}^i & 0  \\
    0 & \mathrm{X}^j 
\end{pmatrix}$}.
\end{align}
By Lemma \ref{lem-ses-reduce}, we have
\begin{align}\label{ex-ij-i-short}
    V_j \xto{{}^t(X^iV_{\alpha}\, X^j) } V^i \oplus V^j \xto{(d_1\,\,d_1'c_1)} C_{ij}(V)_i \to 0.
\end{align}
By combining the exact rows of the right hand diagrams in (\ref{ex-ij-ji}) and (\ref{ex-jij-iji}), and by applying Lemma \ref{lem-ses-reduce}, we have the following exact sequence:
\begin{align}\label{ex-iji-i-short}
    V_j \xto{{}^t(\gamma_2 c_2'\,\,X^j)} V^i \oplus V^j \xto{(e_2\,\,e_2'd_2)} C_{iji}(V)_i \to 0.
\end{align}
By the right diagram of (\ref{ex-j-i}), we have $X^iV_{\alpha}=\gamma_2 c_2'$.
Thus we have $C_{ij}(V)_i=C_{iji}(V)_i$ and $(e_2\,\,e_2'd_2)=(d_1\,\,d_1'c_1)$.

Similarly,  $C_{iji}(V)_j=C_{jij}(V)_j$ holds, where this vector space is given by a cokernel of ${}^t(-X^jV_{\alpha^{\ast}}\,\,X^i) : V_i \to V^j\oplus V^i$.
We have $(e_1\,\,e_1'd_1)=(d_2\,\,d_2'c_2)$.

By the above argument, we have $C_{iji}(V)=C_{jij}(V)$ as vector spaces.
We next show that this equation is compatible with an action of $\Pi^{\la'}$.
Let $b\in\bar{Q}_1$.
We show that $C_{iji}(V)_b=C_{jij}(V)_b$ holds.
If $h(b), t(b)\not\in\{i,j\}$, then $C_{iji}(V)_b=V_b=C_{jij}(V)_b$ holds.

If $h(b)=i$ and $t(b)\neq j$, recall that $d_1=e_2$.
Thus we have $C_{iji}(V)_b=\ep(b)e_2\mu_b=\ep(b)d_1\mu_b=C_{ij}(V)_b=C_{jij}(V)_b$.
In a similar way, the equation holds if $h(b)=j$ and $t(b)\neq i$. 

Assume that $t(b)=i$ and $h(b)\neq j$.
Then $C_{iji}(V)_b=\pi_{b^{\ast}}\ep_2$ and $C_{jij}(V)_b=C_{ij}(V)_b=\pi_{b^{\ast}}\delta_1$ hold.
We have
\begin{align*}
    &\ep_2e_2=\gamma_2c_2-(r_{j}r_{i}\la)_i\id=X^iY^i-\la_i\id-(r_{j}r_{i}\la)_i\id=X^iY^i-(r_j\la)_i\id=\delta_1d_1,\\
    &\ep_2e_2'd_2=\gamma_2\delta_2'd_2=\gamma_2c_2'Y^j=X^iV_{\alpha}Y^j=X^i\gamma_1'c_1=\delta_1d_1'c_1.
\end{align*}
Therefore $\ep_2(e_2\,\,e_2'd_2)=\delta_1(d_1\,\,d_1'c_1)$.
Since $(e_2\,\,e_2'd_2)=(d_1\,\,d_1'c_1)$ is surjective, we have $\ep_2 = \delta_1$.
Thus $C_{iji}(V)_b=\pi_{b^{\ast}}\ep_2=\pi_{b^{\ast}}\delta_1=C_{jij}(V)_b$.
In a similar way, the equation holds if $t(b)=j$ and $h(b)\neq i$. 

Assume that $b=\alpha^{\ast}$.
Then $C_{iji}(V)_b=\ep_2'$ and $C_{jij}(V)_b=-e_1'$.
Since $(e_1\,e_1'd_1)=(d_2\,d_2'c_2)$, we have
\begin{align*}
    \ep_2'e_2 &=-d_2'c_2=-e_1'd_1\\
    \ep_2'e_2'd_2& =(-d_2'\delta_2'-(r_{j}r_{i}\la)_i\id)d_2=-d_2'c_2'Y^j-(r_{j}r_{i}\la)_id_2\\
    &=d_2(X^jY^j-(r_{j}r_{i}\la)_i\id)=e_1(X^jY^j-\la_j\id)\\ &=e_1\gamma_1c_1=-e_1'd_1'c_1.
\end{align*}
Namely, $\ep_2'(e_1\,e_1'd_1)=-e_1'(d_2\,d_2'c_2)$.
Therefore $\ep_2'=-e_1'$.
In a similar way, the equation holds if $b=\alpha$.
This completes the proof that $C_{iji}(V)=C_{jij}(V)$ as $\Pi^{\la'}$-modules.

Finally we show that this identification is a morphism of functors, that is, if $f: U \to V$ is a morphism of $\Pi^{\la}$-modules, then $C_{iji}(f)=C_{jij}(f)$.
Let $k$ be a vertex of $Q$.
If $k\neq i,j$, then $C_{iji}(f)_k=f_k=C_{jij}(f)_k$.
If $k=i$, then $C_{jij}(f)_i=C_{ij}(f)_i$.
Since $C_{ij}(f)_i$ is induced from (\ref{ex-ij-i-short}) and $C_{iji}(f)_i$ is induced from (\ref{ex-iji-i-short}), we have $C_{ji}(f)_i=C_{jij}(f)_i$.
In a similar way, we have $C_{iji}(f)_j=C_{jij}(f)_j$.
This completes the proof of the theorem.
\end{proof}

\begin{proof}[Proof of Theorem \ref{intro-thm-iji-jij}]
The relations on the $C_i$ are shown in the previous theorem.
The relations on the $K_i$ follow, since $K_i$ is right adjoint to $C_i$.
\end{proof}

\section{Tilting ideals for 2-Calabi-Yau algebras}
\subsection{Preliminaries}
\label{sec-tilt-prelim}
We denote by $\sD_{\Fd}(\Mod A)$ the triangulated subcategory of the derived category $\sD(\Mod A)$ which consists of the complexes whose total homology is a finite-dimensional $A$-module.

\begin{pro}\label{pro-2-CY-preproj}
If $A$ is a $d$-Calabi-Yau algebra, then 
\begin{enumerate}[{\rm (a)}]
\item 
There is a functorial isomorphism $\kD\Hom_{\sD(\Mod A)}(M, N) \cong \Hom_{\sD(\Mod A)}(N, M[d])$ for $M$ in $\sD_{\Fd}(\Mod A)$ and $N$ in $\sD(\Mod A)$. 
\item
In particular $\kD\Ext^i_A (M, N) \cong \Ext^{d-i}_A(N, M)$ for $M$ a finite-dimensional $A$-module and $N$ any $A$-module. 
\item 
$\Ext_A^i(M,A)=0$ if $i\neq d$ and $\kD \Ext_A^d(M,A)\cong M$ for a finite-dimensional $A$-module $M$.
\end{enumerate}
\end{pro}

\begin{proof}
For part (a) use \cite[Lemma 4.1]{Kel08}. Parts (b) and (c) are special cases.
\end{proof}

Proposition \ref{pro-ial-props} is a special case of the following result.

\begin{pro}\label{pro-I-nCY}
Let $A$ be a $d$-Calabi-Yau algebra with $d\geq 2$, let $S$ be a finite-dimensional simple left $A$-module and let $I=\Ann_A(S)$.
\begin{enumerate}[{\rm (i)}]
\item
As a left $A$-module, $A/I$ is isomorphic to a finite direct sum of copies of $S$, and as a right $A$-module it is isomorphic to a finite direct sum of copies of $\kD(S)$.
\item
If $\Ext_A^1(S,S)=0$, then $I$ is a faithfully balanced $A$-bimodule, that is, the natural maps
$A\to \End(I_A)$ and $A^{op}\to \End({}_A I)$ given by the homotheties of left and right multiplication, are isomorphisms.
\item
If $\Ext_A^i(S,S)=0$ for $i=1,2,\dots, d-1$, then as a left $A$-module or as a right $A$-module, $I$ is a $(d-1)$-tilting module.
\end{enumerate}
\end{pro}

\begin{proof}
(i) Use that $A/I$ embeds in $\End_K(S) \cong S\otimes_K \kD(S)$.

(ii) By the $d$-Calabi-Yau property, we have $\kD \Ext_A^i(A/I, A) \cong \Ext_A^{d-i}(A,A/I)=0$ for $i<d$.
It follows that the restriction map $\Hom_A(A,A)\to \Hom_A(I,A)$ is an isomorphism.
Since $\Ext^1(S, S)=0$, the restriction map $\Hom_A(A, A/I) \to \Hom_A(I, A/I)$ is surjective,
but since the natural map $\Hom_A(A/I,A/I)\to \Hom_A(A,A/I)$ is onto, we have $\Hom_A(I, A/I)=0$.
This implies that the natural map $\Hom_A(I,I)\to \Hom_A(I, A)$ is an isomorphism.
This gives an isomorphism $A^{op}\to \End({}_A I)$. The other isomorphism follows by symmetry.

(iii) Since $A$ is $d$-Calabi-Yau, the global dimension of $A$ is $d$.
Therefore $I$ has a projective dimension at most $d-1$ on both sides.
It is equal to $d-1$ since
\[
\Ext^{d-1}_A(I,S) \cong \Ext^d_A(A/I,S) \cong \kD \Hom_A(S,A/I) \neq 0.
\]
We show that ${}_A I$ admits a projective resolution with finitely generated projective $A$-modules. Since $A$ is homologically smooth, it has a finite projective resolution by finitely generated $A$-bimodules. By applying $(-)\otimes_A A/I$ to a bimodule resolution of $A$, $A/I$ has a finite projective resolution by finitely generated projective $A$-modules. This implies that ${}_A I$ has a projective resolution with finitely generated projective $A$-modules by Schanuel's lemma, using, for example \cite[Lemma 2.5]{Kimura20}. By symmetry also $I_A$ has a projective resolution with finitely generated projective right $A$-modules.

Next we show that $\Ext_A^i(I,I)=0$ for any $i>0$.
For $i=1,\dots,d-1$, we have
\[
\kD \Ext_A^i(I, I) \cong \kD \Ext_A^{i+1}(A/I,I) \cong \Ext_A^{d-i-1}(I,A/I) \cong \Ext_A^{d-i}(A/I, A/I)=0
\]
where the second isomorphism comes from the $d$-Calabi-Yau property and the last equality comes from the condition that $\Ext_A^i(S, S)=0$ for $i=1,2,\dots, d-1$.
Since the projective dimension of $I$ is at most $d-1$, $\Ext_A^i(I,I)=0$ for any $i>0$.
By symmetry also $\Ext^i(I_A, I_A)=0$ for $i>0$.

Finally, we show that $A$ admits a finite coresolution by modules in $\add({}_A I)$.
Let $0 \to P_{d-1} \to \dots \to P_0 \to I \to 0$ be resolution of $I$ by finitely generated projective right $A$-modules. Applying the functor $\Hom(-,I_A)$ to this resolution, using that $\End(I_A) \cong A$, that $\Ext^i(I_A, I_A)=0$ for $i>0$, and $\Hom(P_i,I_A) \in \add({}_A I)$, we obtain the desired exact sequence.
By symmetry also $A$ admits a finite coresolution by modules in $\add(I_A)$.
\end{proof}

The following lemma is standard. See for example \cite[Lemma III.1.1]{BIRS} in the Krull-Schmidt case and \cite[Lemma 2.8]{SY13} in the tilting module case.

\begin{lem}
\label{lem-SY-Tor}
If $T$ is a partial tilting $A$-module and $S$ is a finite-dimensional simple right $A$-module, then at least one of $S\otimes_A T=0$ or $\Tor_1^A(S,T)=0$ holds.
\end{lem}

\begin{proof}
Since $\kD(S)$ is a simple left module, either $\Hom_A(T,\kD(S))=0$ or there is a surjection $T\to \kD(S)$. In the second case, applying $\Hom_A(T,-)$ and using that $T$ is partial tilting, one obtains that $\Ext^1_A(T,\kD(S))=0$. Rewriting in terms of tensor products and $\Tor$ gives the result.
\end{proof}

\subsection{Tilting ideals}
\label{sec-tilt-ideals}
For an $A$-module $S$, let $I_S=\Ann_A(S)$.
In the proof of the following proposition, we refer the proof of \cite[Proposition III.1.5]{BIRS}.

\begin{pro}\label{pro-IT-tilting}
Let $A$ be 2-Calabi-Yau, let $T$ be a tilting module for $A$ and let $S$ be a finite-dimensional rigid simple $A$-module.
\begin{enumerate}[{\rm (a)}]
    \item We have $I_S \otimes_A^{\mbL}T \cong I_S \otimes_A T$ in the derived category of $A$-modules.
    \item If $\Tor_1^A(\kD(S), T)=0$, then the natural map $I_S \otimes_A T \to I_S T$ is an isomorphism.
    \item $I_S T$ is a tilting module for $A$ and $\End_A(I_S T) \cong \End_A(T)$.
\end{enumerate}
\end{pro}

\begin{proof}
(a) We have a short exact sequence 
\begin{align}\label{ses-I-P-P/I}
0 \to I_S \to A \to A/I_S \to 0.
\end{align}
Applying $-\otimes_A T$ to this short exact sequence, we have $\Tor_1^A(I_S, T)\cong \Tor_2^A(A/I_S,T)=0$, giving the result.

(b) We have an exact sequence 
\[
\Tor_1^A(A/I_S, T) \to I_S\otimes_A T \to T \to (A/I_S)\otimes_A T \to 0,
\]
and $A/I_S$ is isomorphic to a finite direct sum of copies of $\kD(S)$ as a right $A$-module, so if $\Tor_1^A(\kD(S), T)=0$, then the map $I_S\otimes_A T \to I_S T$ is an isomorphism.

(c) By Lemma \ref{lem-SY-Tor} we have $\kD(S)\otimes_A T=0$ or $\Tor_1^A(\kD(S), T)=0$. In the first case, $I_S T=T$ and the assertion is trivial. Thus we may assume that $\Tor_1^A(\kD(S), T)=0$. Now the projective dimension of $T$ is at most one and the projective dimension of $T/I_S T$ is at most two, since the global dimension of $A$ is two. Thus the projective dimension of $I_S T$ is at most one. By (a) and (b) we have $I_S T \cong I_S \otimes_A^{\mbL}T$. Since $I_S$ is a tilting ideal and $T$ is a tilting module, this is tilting complex; see \cite[Lemma III.1.2(a)]{BIRS}. Now the condition on projective dimension implies that $I_S T$ is a tilting module.
\end{proof}

\begin{proof}[Proof of Theorem~\ref{thm-set-i-equiv}]
We show by induction on $r$ that $I_{S_1 S_2\dots S_r}=I_{S_1}I_{S_2}\cdots I_{S_r}$ is a tilting ideal of finite codimension in $A$.
This is clear for $r=1$. If we know it for $r-1$, then
$I_{S_2\dots S_r}$ is a tilting ideal of finite codimension.
Now
\[
I_{S_2\dots S_r}/I_{S_1 S_2\dots S_r} \cong (A/I_{S_1}) \otimes_A I_{S_2\dots S_r}
\]
which is finite dimensional since $A/I_{S_1}$ is finite dimensional and $I_{S_2\dots S_r}$ is a finitely generated left $A$-module. Also $I_{S_1\dots S_r} = I_{S_1} I_{S_2\dots S_r}$ is a tilting module by Proposition~\ref{pro-IT-tilting}. Similarly, $I_{S_1\dots S_r} = I_{S_1\dots S_{r-1}} I_{S_r}$ is tilting as a right $A$-module.

Now let $I$ be a partial tilting left ideal in $A$ with $A/I\in\mcE(\mcS)$. We show by induction on $\dim_K (A/I)$ that $I\in \mcI(\mcS)$. If $I\neq A$, choose a simple submodule $S$ of $A/I$.

We know $\kD \Hom_A(S,A) \cong \Ext^2_A(A,S) = 0$, so $\Hom_A(S,A)=0$.
It follows that $\Ext^1(S,I)\neq 0$. Thus $\Tor_1^A(\kD(S),I)\neq 0$.
Thus by Lemma~\ref{lem-SY-Tor}, $\kD(S)\otimes_A I=0$.
Thus $(A/I_S)\otimes_A I = 0$, so $I_S I = I$.

Let $U = \{ a\in A \mid I_S a \subseteq I\}$. It is a left ideal with $I_S U \subseteq I \subseteq U$. Since $I_S I = I$ it follows that $I_S U=I$.

Since the natural map $A\to \Hom_A(I_S,A)$ is an isomorphism, we have $U\cong \Hom_A(I_S,I)$.
Since $\Ext^1_A(I_S, I) \cong \Ext^2_A(A/I_S, I) \cong \kD\Hom_A(I_SI, A/I_S)=0$, we have $U \cong \Hom_A(I_S, I) \cong \RHom_A(I_S, I)$.
This implies that $U$ is a partial tilting module.
Since there is a surjection from $A/I_S U$ to $A/U$, $A/U$ belongs to $\mcE(\mcS)$.
By induction, $U$ is in $\mcI(\mcS)$ and so is $I=I_SU$.

Finally, if $I,I'\in\mcI(\mcS)$ are isomorphic as $A$-modules, then $I=I'$ by the argument of \cite[Theorem III.1.6(d)]{BIRS}.
Namely, by the 2-Calabi-Yau property $\Ext^1_A(A/I,A)=0$, so an isomorphism $f:I\to I'$ lifts to a map $f':A\to A$. Thus there is an element $a\in A$ such that $f'$ is right multiplication by $a$. Then $I'=\Im f = Ia \subseteq I$,
and by symmetry $I\subseteq I'$.
\end{proof}

\subsection{Relations}
\label{sec-relns}
\begin{lem}\label{lem-two-simple-serre}
Let $A$ be 2-Calabi-Yau and let $S,T$ be finite-dimensional rigid simple $A$-modules such that 
$\Ext^1_A(S,T)$ is 1-dimensional as a right $\End_A(S)$-module and as a left $\End_A(T)$-module.
Let $E_{S}^{T}$ (resp. $E^{S}_{T}$) be the unique $A$-module of length two such that the top is $T$ (resp. $S$) and the socle is $S$ (resp. $T$).
We denote by $\mcE(\{S,T\})$ the Serre subcategory of the category of finite-dimensional $A$-modules generated by $S$ and $T$.
Then the following statements hold.
    \begin{enumerate}[{\rm (a)}]
    \item $\mcE(\{S,T\})$ is a uniserial category, that is, any indecomposable module in $\mcE$ has a unique composition series.
    \item If $L\in \mcE(\{S,T\})$ is indecomposable, then $L$ is isomorphic to one of $S$, $T$, $E_{S}^{T}$ and $E^{S}_{T}$.
\end{enumerate}
\end{lem}

\begin{proof}
(a) By assumption $\Ext^1_A(S,T)$ is 1-dimensional as a right $\End_A(S)$-module and as a left $\End_A(T)$-module. By the 2-Calabi-Yau property, also $\Ext^1_A(T,S)$ is 1-dimensional as a right $\End_A(T)$-module and as a left $\End_A(S)$-module. The claim then follows from \cite[subsection 8.3]{Gabriel73}.

(b) Let $M$ be an indecomposable module in $\mcE$ of length $n$. Without loss of generality its top is $S$. If $n\geq 3$, since $\mcE(\{S,T\})$ is uniserial, there is an indecomposable factor module $N$ of $M$ of length three. Then there is a non-split short exact sequence $0 \to S \to N \to E_{T}^S \to 0$. But applying $\Hom_A(S,-)$ to this sequence and using that $\Hom_A(S,E_T^S)=0$, $\Ext^1_A(S,S)=0$ and that $\Ext^1_A(S,T)$ is 1-dimensional as a right $\End_A(S)$-module, we see that $\Ext^1_A(S,E^S_T)=0$. Thus by the 2-Calabi-Yau property $\Ext^1_A(E^S_T,S)=0$, a contradiction. Thus any indecomposable module in $\mcE$ has length at most two, and the assertion holds.
\end{proof}

For a subcategories $\mcB$ and $\mcC$ of an abelian category $\mcA$, we denote by $\mcB\ast \mcC$ the subcategory of $\mcA$ consisting objects $A\in\mcA$ admitting a short exact sequence $0\to B \to A \to C\to 0$ in $\mcA$ with $B\in \mcB$ and $C\in\mcC$.

\begin{pro}\label{prop-addS-relation}
Let $A$ be 2-Calabi-Yau. For finite-dimensional rigid simple $A$-modules $S,T$, the following hold.
\begin{enumerate}[{\rm (a)}]
    \item $\add(S)\ast\add(S)=\add(S)$,
    \item $\add(S)\ast\add(T) = \add(T)\ast\add(S)$ if $\Ext^1_A(S,T)=0$,
    \item $\add(S)\ast\add(T)\ast\add(S) = \add(T)\ast\add(S)\ast\add(T)$ if $\Ext^1_A(S,T)$ is 1-dimensional as a right $\End_A(S)$-module and as a left $\End_A(T)$-module.
\end{enumerate}
\end{pro}

\begin{proof}
Part (a) holds since $\Ext_A^1(S,S)=0$, and (b) since $\Ext_A^1(S,T)=\Ext^1_A(T,S)=0$. Part (c) follows from Lemma \ref{lem-two-simple-serre}.
\end{proof}

\begin{proof}[Proof of Proposition~\ref{pro-ideal-relation}]
For a finite-dimensional $A$-module $M$ and a sequence $S_1, S_2,\dots, S_r$ of rigid finite dimensional simple modules, we have that $I_{S_1\dots S_r}M=0$ if and only if $M\in\add S_1 \ast \dots \ast\add S_r$. Now parts (a)-(c) follow from the corresponding parts of Proposition~\ref{prop-addS-relation}. For example in case (c), the module $M = A/I_{STS}$ is finite-dimensional by Theorem~\ref{thm-set-i-equiv}, and annihilated by $I_{STS}$, so is in $\add(S)\ast\add(T)\ast\add(S)$. Thus $M$ is in $\add(T)\ast\add(S)\ast\add(T)$ by Proposition~\ref{prop-addS-relation}(c), so $I_{TST} M =0$, and hence
$I_{TST}\subseteq I_{STS}$. Similarly $I_{STS}\subseteq I_{TST}$, giving equality.
\end{proof}

Now let $\mcS$ be a set of pairwise non-isomorphic finite-dimensional rigid simple $A$-modules. For simplicity we assume that $\mcS$ is split, meaning that $\End_A(S)=K$ for all $S\in\mcS$. We denote by $\mcE(\mcS)$ the Serre subcategory of the category of finite-dimensional $A$-modules generated by $\mcS$. Moreover we denote by $\sD_{\mcE(\mcS)}(\Mod A)$ the triangulated subcategory of $\sD(\Mod A)$ consisting of the objects $X$ such that the total cohomology of $X$ belongs to $\mcE(\mcS)$, that is, $\bigoplus_{n\in\mbZ}H^n(X)\in\mcE(\mcS)$.

Let $\mcZ=\bigoplus_{S\in\mcS}\mbZ e_S$ be the free $\mbZ$-module with basis elements $e_S$, where $S$ runs through the elements of $\mcS$, up to isomorphism, 
and let $(-,-)$ be the bilinear form on $\mcZ$ given by
\[
(e_S,e_T) = \sum_{i\ge 0} (-1)^i \dim_K \Ext^i_A(S,T).
\]
It is defined and symmetric for $A$ 2-Calabi-Yau.
For $M\in \mcE(\mcS)$, let
\[
[M]:=\sum_{S\in\mcS}(M : S)  e_S \in \mcZ,
\]
where $(M : S)$ is the multiplicity of $S$ as a composition factor of $M$.
For $X$ in $\sD_{\mcE(\mcS)}(\Mod A)$, we define
\[
F(X) = \sum_{n\in\mbZ}(-1)^n[H^nX] \in \mcZ.
\]
Let $W(\mcS)$ be the Coxeter group, as in the introduction.
We define an action of $W(\mcS)$ on $\mcZ$ by $\si_S(x):=x-(e_S,x) e_S$.

\begin{pro}\label{pro-I-S-tensor}
Let $S, T\in\mcS$ and let $m=\dim_K \Ext^1_A(S,T)$.
\begin{enumerate}[{\rm (a)}]
\item $\Tor_1^A(I_S, S)\cong S$ as left $A$-modules.
\item If $S\not\cong T$, then the $A$-module $I_S\otimes_A T$ has exactly two composition factors $S$ and $T$, with multiplicities $m$ and $1$, respectively.
\item 
We have
\[
  I_S\otimes_A^{\mbL}T\cong\begin{cases}
    I_S\otimes_A T & (S \not\cong T), \\
    \Tor_1^A(I_S, T)[1] & (S\cong T).
  \end{cases}
\]
\item The functor $I_S\otimes_A^{\mbL} -$ sends $\sD_{\mcE(\mcS)}(\Mod A)$ to itself, and $F(I_S\otimes_A^{\mbL} X) = \si_S (F(X))$ for $X$ in $\sD_{\mcE(\mcS)}(\Mod A)$.
\end{enumerate}
\end{pro}
\begin{proof}
By assumption $\End_A(S)=K$, so by the Jacobson Density Theorem, $A/I_S\cong \End_K(S) \cong S\otimes_K \kD(S)$ as $A$-bimodules. Thus also $\kD(A/I_S)\cong S\otimes_K \kD(S)$. For $i\ge 0$, we have 
\[
\Tor_i^A(A/I_S, T) \cong \kD\Ext^i_A(T,\kD(A/I_S))
\cong \kD\Ext^i_A(T,S\otimes_K \kD(S))
\]
\[
\cong \kD\left(\Ext^i_A(T,S)\otimes_K \kD(S)\right)
\cong S \otimes_K \kD\Ext^i_A(T,S)
\]
as left $A$-modules. Thus, tensoring the exact sequence~(\ref{ses-I-P-P/I}) with $T$, gives an exact sequence
\begin{align}\label{seq-T-IS-S-SSS}
0 \to 
S \otimes_K \kD\Ext^1_A(T,S) \to I_S \otimes_A T \to T \to S \otimes_K \kD\Hom_A(T,S)
\to 0,
\end{align}
as well as $\Tor_1^A(I_S, T) \cong \Tor_2^A(A/I_S, T) \cong S \otimes_K \kD\Ext^2_A(T,S)$. 

(a) Taking $S=T$, by the $2$-Calabi-Yau property, $\Ext_A^2(S,S)$ is one dimensional over $K$, so $\Tor_1^A(I_S, S)$ is isomorphic to $S$.

(b) Follows from the exact sequence (\ref{seq-T-IS-S-SSS}).

(c) Follows from (a), (b) and Lemma \ref{lem-SY-Tor}.

(d) Follows from (c).
\end{proof}

The following proposition is the $\Hom$ version of Proposition \ref{pro-I-S-tensor}.
We omit the proof since it is essentially the same as the tensor version.

\begin{pro}\label{pro-I-S-hom}
Let $S,T\in\mcS$ and let $m=\dim_K \Ext^1_A(S,T)$.
\begin{enumerate}[{\rm (a)}]
\item $\Ext_A^1(I_S, S) \cong S$ as left $A$-modules.
\item If $S\not\cong T$, then the $A$-module $\Hom_A(I_S, T)$ has exactly two composition factors $S$ and $T$, with multiplicities $m$ and $1$, respectively.
\item We have
\[
  \RHom_A(I_S, T)\cong\begin{cases}
    \Hom_A (I_S, T) & (S\not\cong T), \\
    \Ext^1_A(I_S, T)[1] & (S\cong T).
  \end{cases}
\]
\item The functor $\RHom_A(I_S, -)$ sends $\sD_{\mcE(\mcS)}(\Mod A)$ to itself, and $F(\RHom_A(I_S, X)) = \si_S (F(X))$ for $X$ in $\sD_{\mcE(\mcS)}(\Mod A)$.
\end{enumerate}
\end{pro}

\begin{proof}[Proof of Theorem~\ref{thm-cox-bij}]
We adapt the argument of \cite[Theorem III.1.9]{BIRS}. By \cite[Theorem 3.3.1(ii)]{BB05} one can pass between any two reduced expressions for $w$ by the operations of (i) replacing $\sigma_S \sigma_T$ with $\sigma_T \sigma_T$, if there is no arrow in $Q(\mcS)$ from $S$ to $T$, and (ii) replacing $\sigma_S \sigma_T \sigma_S$ with $\sigma_T \sigma_S \sigma_T$ if there is a unique arrow in $Q(\mcS)$ from $S$ to $T$. Thus by Proposition~\ref{pro-ideal-relation} the mapping is well-defined. 

To show the mapping is surjective, let $I \in \mcI(\mcS)$, take an expression $I = I_{S_1}I_{S_2}\dots I_{S_k}$ with $k$ minimal and let $w = \sigma_{S_1}\dots \sigma_{S_k}$. By \cite[Theorem 3.3.1(i)]{BB05}, one can pass to a reduced expression for $w$ by the operations (i) and (ii), as above, and (iii) remove $\sigma_S \sigma_S$. By Proposition~\ref{pro-ideal-relation}(i) and the minimality of $k$, operation (iii) doesn't occur, so the expression $w = \sigma_{S_1}\dots \sigma_{S_k}$ is reduced.

To prove that the mapping is injective, we may assume that $\mcS$ contains only finitely many isomorphism classes of rigid simple modules, for if two Coxeter group elements $w,w'$ with reduced expressions involving the generators $\si_{S_1},\dots,\si_{S_\ell}$ are sent to the same ideal, then they are equal in $W(\{S_1,\dots,S_\ell\})$, so also equal in $W(\mcS)$. Now the action of $W(\mcS)$ on $\mcZ$ is the `geometric representation' considered in \cite[section 4.2]{BB05}. 

We claim that if $w = \sigma_{S_1}\dots \sigma_{S_k}$ is a reduced expression, then $I = I_{S_1}\dots I_{S_k} \cong I_{S_1}\otimes_\Lambda^{\mbL}\dots \otimes_\Lambda^{\mbL} I_{S_k}$. Then $F(I\otimes_A^\mbL X) = w(F(X))$ for $X \in \sD_{\mcE(\mcS)}(\Mod A)$ by Proposition~\ref{pro-I-S-tensor}(d), and since the geometric representation is faithful by \cite[Theorem 4.2.7]{BB05}, it follows that $I$ determines~$w$.

We prove the claim by induction on $k$, so let $w' = \sigma_{S_1}\dots \sigma_{S_{k-1}}$ and $I' = I_{S_1}\dots I_{S_{k-1}}$. By Proposition~\ref{pro-IT-tilting}, for the opposite algebra, and with $S = D(S_k)$, it suffices to show that $\Tor_1^A(I',S_k)=0$. By Lemma~\ref{lem-SY-Tor}, for the opposite algebra, if this fails, then $I'\otimes_A S_k = 0$. Thus it suffices to show that $F(I'\otimes_A^\mbL S_k)$ is positive, but by the induction this is $w'(F(S_k))$, and this is positive by \cite[Proposition 4.2.5(i)]{BB05}.
\end{proof}

\subsection{Extended Dynkin quivers}
\label{sec-ext-dynkin}
In this subsection $Q$ is an extended Dynkin quiver and $\delta$ is the minimal positive imaginary root.

\begin{proof}[Proof of Theorem~\ref{thm-all-cof}]
Let $X$ be the set of roots for $Q$ together with zero. If $\alpha\in X$ then the orbit $\{ \alpha+n\delta \mid n\in\mbZ \}$ under addition of $\delta$ is a subset of $X$, and since any orbit contains a root for the Dynkin quiver, the set of orbits is finite. But each orbit can contain at most one element of $\Sigma_\la^{\rm re}$, for if $\alpha$ and $\beta = \alpha+m\delta$ belong to $\Sigma_\la^{\rm re}$, with $m>0$, then the decomposition $\beta=\alpha+\gamma$ with $\gamma=m\delta$ contradicts that $\beta\in \Sigma_\la^{\rm re}$.

Suppose $I$ is a tilting ideal in $A= \Pi^\la(Q)$ with $A/I$ finite dimensional. If $S$ is a finite-dimensional rigid simple module for $A= \Pi^\la(Q)$, and $T$ is a finite-dimensional simple module which is not rigid, then $\Dim T$ is an imaginary root, so a multiple of $\delta$, so $(\Dim S,\Dim T)=0$. Thus $\Ext^1_A(S,T)=\Ext^1_A(T,S)=0$ by Proposition~\ref{pro-dim-Ext}. It follows that $A/I$ decomposes as a direct sum $Y\oplus Z$ where $\Dim Y$ is a multiple of $\delta$ and $Z$ has composition factors in $\mcR$. Now if $Y\neq 0$, then $\End_A(Y)\neq 0$, so $\Ext^1_A(Y,Y)\neq 0$ by Proposition~\ref{pro-dim-Ext}. Thus $\Ext^1_A(A/I,A/I)\neq 0$. On the other hand, since $\Ext^1_A(I,I)=0$ we have a commutative diagram with exact rows
\[
\begin{CD}
0 @>>> \Hom_A(A,I) @>>> \Hom_A(A,A) @>>> \Hom_A(A,A/I) @>>> 0 \\
& & @VVV @VfVV @VgVV \\
0 @>>> \Hom_A(I,I) @>>> \Hom_A(I,A) @>>> \Hom_A(I,A/I) @>>> 0.
\end{CD}
\]
Thus there is an induced surjection from $\Coker f$ to $\Coker g$. Now we have $\kD \Coker f\cong \kD \Ext^1_A(A/I,A) \cong \Ext^1_A(A,A/I) = 0$ by the 2-Calabi-Yau property, and $\Coker g\cong \Ext^1_A(A/I,A/I)$, so this vanishes, a contradiction. Thus $Y=0$, so $A/I$ has composition factors in $\mcR$.
\end{proof}

Now suppose that $\la\cdot\delta=0$, let $\Sigma_\la^{\rm re} = \{ \al_1,\dots,\al_k \}$ and let $S_{\al_i}$ be the rigid simple module of dimension vector~$\al_i$. We have $\al_i<\delta$, for $\al_i-\delta$ is a real root with $\la\cdot(\al_i-\delta)=0$, and it can't be positive, since that would contradict that $\al_i\in\Sigma_\la^{\rm re}$.

Let $Q(\mcR)$ be the Ext-quiver of $\mcR$, so with vertex set $\{1,\dots,k\}$ and the number of arrows from $i$ to $j$ is zero if $i=j$ and is $-(\al_i,\al_j)_Q$ if $i\neq j$. Here $(-,-)_Q$ is the symmetric bilinear form for $Q$, and we write $q_Q$ for the corresponding quadratic form. Now $Q(\mcR)$ is the double of some quiver $\Ga$, that is, $\bar{\Ga}=Q(\mcR)$. We denote by $q_{\Ga}$ the quadratic form for $\Ga$.

\begin{pro}\label{pro-comp-ext}
The following statements hold.
\begin{enumerate}[{\rm (a)}]
    \item We have $q_{\Ga}(d)=q_Q(\Sigma_{i=1}^k d_i\al_i)$ for $d=(d_i)\in\mbZ^k$.
    \item $\Ga$ is a disjoint union of extended Dynkin quivers.
\end{enumerate}
\end{pro}
\begin{proof}
(a)
For $d=(d_i)\in\mbZ^k$, we have
\[
2q_Q(\Sigma_{i=1}^kd_i\al_i) = \Sigma_{i,j}(d_i\al_i, d_j\al_j)_Q = \Sigma_{i,j}d_id_j(\al_i,\al_j)_Q = 2q_{\Ga}(d),
\]
where $(-,-)_Q$ is the symmetric bilinear form of $Q$.
Thus we have the assertion.

(b) By (a), $q_{\Ga}$ is positive semi-definite, since so is $q_Q$. Therefore $\Ga$ is a disjoint union of Dynkin quivers and extended Dynkin quivers. We show that the connected components of $\Ga$ are not Dynkin. For $\al_j\in \Sigma_\la^{\rm re}$, the element $\beta = \delta-\al_j$ is a positive real root and $\la\cdot\beta=0$. Consider a decomposition $\beta=\gamma_1+\dots+\gamma_r$ with the $\gamma_i$ positive (and necessarily real) roots with $\la\cdot\gamma_i=0$, and with $r$ as large as possible. The maximality implies that $\gamma_i\in\Sigma_\la^{\rm re}$, so each $\gamma_i$ is an $\al_\ell$ for some $\ell$. Collecting terms, this implies that for any vertex $j$ of $\Ga$, there is a dimension vector $d=(d_i)\in\mbN^k$ such that $d_j>0$ and $\delta=\Sigma_{i=1}^k d_i\al_i$. By (a), we have $q_{\Ga}(d)= q_Q(\delta)=0$. Thus connected component of $\Ga$ containing vertex $j$ cannot be Dynkin.
\end{proof}

\begin{lem}\label{lem-qQ-qGa}
For $d\in \mbZ^k$, $\Sigma_{i=1}^k d_i\al_i$ is a multiple of $\delta$ if and only if the restriction of $d$ to each connected component of $\Ga$ is a multiple of the radical vector for that component.
\end{lem}
\begin{proof}
This assertion directly follows from Proposition \ref{pro-comp-ext}.
\end{proof}
\begin{lem}\label{lem-del-del}
Let $d=(d_i)\in\mbN^k$ such that $\Sigma_{i=1}^kd_i\al_i=\delta$ with $d_1>0$.
Let $\Ga'$ be the connected component of $\Ga$ containing the vertex $1$.
We denote by $\Ga'_0=\{1,\dots, s\}$ the set of vertices of $\Ga'$.
\begin{enumerate}[{\rm (a)}]
    \item We have $d_i=0$ for $i>s$.
    \item The vector $(d_1, d_2, \dots, d_s)$ is the minimal positive imaginary root $\delta'$ for $\Ga'$.
\end{enumerate}
\end{lem}
\begin{proof}
By Lemma \ref{lem-qQ-qGa}, a vector $(d_1, d_2, \dots, d_s)$ is a multiple of $\delta'$, so write $(d_1, d_2, \dots, d_s)=\ell \delta'$ for an integer $\ell>0$. Again by Lemma \ref{lem-qQ-qGa}, $\Sigma_{i=1}^s d_i\al_i$ is a multiple of $\delta$. Since $\delta$ is minimal, we have $\delta=\Sigma_{i=1}^s d_i\al_i$ and $d_i=0$ for $i>s$. Moreover by Lemma \ref{lem-qQ-qGa}, $\Sigma_{i=1}^s \delta'_i\al_i$ is a multiple of $\delta$, say $m \delta$ for an integer $m>0$. Then $\delta=\ell \, \Sigma_{i=1}^s \delta'_i \al_i = \ell m \delta$. Thus $\ell=m=1$ and the assertion holds.
\end{proof}

\begin{proof}[Proof of Theorem~\ref{thm-sing-comp-a}]
Part is already proved in Proposition~\ref{pro-comp-ext}.
Let $\Ga'$ be a connected component of $\Ga$ such that $\Ga'_0=\{1,\dots,s\}$ with minimal imaginary positive root $\delta'$. By Lemma \ref{lem-del-del}, the semisimple $\Pi^\la(Q)$-module $\bigoplus_{i=1}^s S_{\al_i}^{\oplus \delta'_i}$ has dimension vector $\delta$. Recall that, for $K$ algebraically closed of characteristic zero, the elements of $\Rep(\Pi^\la(Q), \al)\dbslash\GL(\al)$ are in 1:1 correspondence with the isomorphism classes of semisimple $\Pi^\la(Q)$-modules of dimension vector $\delta$, and by \cite[Theorem 3.2]{B-02}, the non-singular points correspond to the simple modules. Thus this semisimple module represents a singular point of $\Rep(\Pi^\la(Q), \delta)\dbslash\GL(\delta)$. This argument defines a map from the set of connected components of $\Ga$ to the set of singular points of $\Rep(\Pi^\la(Q), \delta)\dbslash\GL(\delta)$.
We denote this map by $\Phi$.

Conversely, let $x\in\Rep(\Pi^\la(Q), \delta)\dbslash\GL(\delta)$ be a singular point. Then $x$ corresponds to a semisimple $\Pi^\la(Q)$-module $M_x=\bigoplus_{i=1}^k S_{\al_i}^{\oplus d_i}$ with $\Sigma_{i=1}^k d_i\al_i=\delta$.
We may assume that $d_1>0$.
Let $\Ga'$ be a connected component of $\Ga$ containing the vertex $1$ with $\Ga'_0=\{1, \dots, s\}$.
By Lemma \ref{lem-del-del}, we have $d_i=0$ for $i>s$.
Namely, $x$ determines a unique connected component $\Ga'$.
We denote this map by $\Psi$.

It is easy to see that $\Psi$ and $\Phi$ are mutually inverse.
\end{proof}


\end{document}